\documentclass[12pt]{amsart}

\usepackage{amsmath,amssymb,latexsym,amsthm,newlfont,enumerate, longtable}

\usepackage{graphicx}
\usepackage[all]{xy}
\makeindex
\theoremstyle{plain}

\newtheorem{thm}{Theorem}[section]

\newtheorem{pro}[thm]{Proposition}

\newtheorem{lem}[thm]{Lemma}

\newtheorem{cla}[thm]{Claim}
\newtheorem{cor}[thm]{Corollary}
\newtheorem{con}{Conjecture}

\newtheorem{case}{Case}

\theoremstyle{definition}

\newtheorem{dfn}[thm]{Definition}

\newtheorem{nt}[thm]{Notation}

\newtheorem{rem}[thm]{Remark}

\newtheorem{exa}[thm]{Example}

\theoremstyle{remark}

\newtheorem{que}[thm]{Question}

\newcommand{\Z}{\mathbb{Z}}

\newcommand{\C}{\mathbb{C}}

\newcommand{\Q}{\mathbb{Q}}
\newcommand{\PS}{\mathbb{P}}
\newcommand{\F}{\mathbb{F}}

\DeclareMathOperator{\Gal}{Gal}
\DeclareMathOperator{\Pic}{Pic} 
\DeclareMathOperator{\rk}{rk}

\DeclareMathOperator{\codim}{Codim}
\DeclareMathOperator{\Exc}{Exc} \DeclareMathOperator{\Spec}{Spec}

\DeclareMathOperator{\Bs}{Bs}
\DeclareMathOperator{\Def}{Def}

\DeclareMathOperator{\Proj}{Proj}
\DeclareMathOperator{\im}{Im}

\DeclareMathOperator{\Cl}{Cl}

\begin{document}

\bibliographystyle{alpha}

\title[Classification of terminal quartics and rationality]{A classification of terminal quartic $3$-folds and applications to rationality questions}
\date{}
\author{Anne-Sophie Kaloghiros}
\address{Department of Pure Mathematics and Mathematical Statistics, 
Uni\-ver\-si\-ty of Cambridge, Wilberforce Road, Cambridge CB3 0WB, Uni\-ted
Kingdom}
\email{A.S.Kaloghiros@dpmms.cam.ac.uk}
\setcounter{tocdepth}{1}
\maketitle
\begin{abstract}
This paper studies the birational geometry of terminal Gorenstein Fano $3$-folds. If $Y$ is not $\Q$-factorial, in most cases,  it is possible to describe explicitly the divisor class group $\Cl Y$ by running a Minimal Model Program (MMP) on $X$, a small $\Q$-factorialisation of $Y$. In this case, the generators of $\Cl Y/ \Pic Y$ are ``topological traces " of $K$-negative extremal contractions on $X$. One can show, as an application of these methods, that a number of families of non-factorial terminal Gorenstein Fano $3$-folds are rational. In particular, I give some examples of rational quartic hypersurfaces $Y_4\subset \PS^4$ with $\rk \Cl Y=2$  and show that when $\rk \Cl Y\geq 6$, $Y$ is always rational. 
\end{abstract}

\tableofcontents
\section*{Introduction}
\label{sec:introduction}

Let $Y_4 \subset \PS^4$ be a quartic hypersurface in $\PS^4$ with terminal singularities.
The Grothendieck-Lefschetz theorem states that every Cartier divisor on $Y$ is the restriction of a Cartier
divisor on $\PS^4$. Recall that a variety $Y$ is  $\Q$-factorial when a multiple of every Weil divisor is Cartier. There is no analogous statement for the group of Weil divisors: \cite{Kal07b} bounds the rank 
of the divisor class group $\Cl Y$ of $Y_4\subset \PS^4$,  but when $Y_4$ is not factorial, $\Cl Y$ remains poorly understood. 
 Terminal quartic hypersurfaces in $\PS^4$
are terminal Gorenstein Fano $3$-folds.
Any terminal Gorenstein Fano $3$-fold $Y$ is a $1$-parameter
flat deformation of a nonsingular Fano $3$-fold $\mathcal{Y}_{\eta}$  with
$\rho(Y)= \rho(\mathcal{Y}_{\eta})$ \cite{Nam97a}. Nonsingular Fano $3$-folds are classified  in \cite{Isk77, Isk78} and
\cite{MM82, MM03}; there are $17$ deformation families with Picard rank $1$. Terminal Fano $3$-folds with
$\Q$-factorial singularities play a central role in Mori theory: they are one of the possible end products of the
Minimal Model Program (MMP) on nonsingular varieties. In \cite{T02}, Takagi develops a method to classify such $\Q$-Fano $3$-folds under some mild assumptions; his techniques rely in an essential way on the sudy of Weil non-Cartier divisors on some canonical Gorenstein Fano $3$-folds.  

By definition,
$\Q$-factoriality is a global topological property:  it depends on the prime divisors lying on $Y$ rather than on
the local analytic type of its singular points alone. The divisor class group of a terminal Gorenstein $3$-fold $Y$ is
torsion free \cite{Kaw88}, so that $Y$ is $\Q$-factorial precisely when it is is factorial.  If $Y$ is a
terminal Gorenstein Fano $3$-fold, by Kawamata-Viehweg Vanishing $\Pic Y\simeq H^2(Y, \Z)$, and  by \cite[Theorem 3.2]{NS95}$, \Cl Y\simeq H_4(Y,\Z)$. Hence $Y$ is factorial if and only if
\[
H_4(Y, \Z)\simeq H^2(Y, \Z).
\] 

Birational techniques can be used to bound the rank of the divisor class group of terminal Gorenstein Fano $3$-folds \cite{Kal07b}. More precisely, under the assumption that $Y$ does not contain a plane, Weil
non-Cartier divisors on $Y$ are precisely the divisors that are contracted by the MMP on a small factorialisation
$X\to Y$. The assumption that $Y$ does not contain a plane guarantees that each step of the MMP on $X$ is a
Gorenstein weak Fano $3$-fold, and hence that this MMP can be studied explicitly. 
In \cite{Kal07b}, I use numerical constraints associated to the extremal contractions of this MMP to bound the
Picard rank of $X$, or equivalently the rank of $\Cl Y$.
These methods can be refined to describe explicitly the possible extremal divisorial contractions that occur in
this MMP. 
For instance, if $Y_4 \subset \PS^4$ a terminal quartic hypersurface, one can state a geometric ``motivation'' of
non-factoriality  as follows:

\begin{thm}[Main Theorem]
  \label{thm:1}
Let $Y_4^3 \subset \PS^4$ be a terminal Gorenstein quartic $3$-fold. Then one of the following holds:
\begin{enumerate} 
\item[1.] $Y$ is factorial.
\item[2.] $Y$ contains a plane $\PS^2$.
\item[3.] $Y$ contains an anticanonically embedded del Pezzo surface of
  degree $4$ and $\rk \Cl Y=2$. 
\item[4.] A small factorialisation of $Y$ has a structure of Conic Bundle over $\PS^2$, $\F_0$
  or $\F_2$ and $\rk \Cl Y= 2$ or $3$.
\item[5.] $Y$ contains a rational scroll $E \to C$  over a curve $C$ whose
  genus and degree appear in Table~\ref{table1} (see page \pageref{table1}).
\end{enumerate}  
\end{thm}

Analogous results can be stated for any terminal Gorenstein Fano $3$-fold (see Section~\ref{tables}).
When $Y$ is not factorial, either all small factorialisations of $Y$ are Mori fibre spaces or at least one of the surfaces listed in Theorem~\ref{thm:1} is a generator of
$\Cl Y/\Pic Y$; observe that these surfaces have relatively low degree (see Remark~\ref{degree}). 
Further, when $Y$ does not contain a plane, the rank of $\Cl Y$ can only be large when $Y$ contains many
independent surfaces of low degree.
The divisor class group is generated by prime divisors that are ``topological traces'' on $Y$ of the extremal
divisorial contractions that occur when running the MMP on $X$. 

The explicit study of the MMP on $X$ exhibits birational models of $Y$ that are small modifications of terminal
Gorenstein Fano $3$-folds with Picard rank $1$ and higher anticanonical degree. Questions on rationality of $Y$,
or at the other end of the spectrum, questions on rigidity, can be easier to answer on these models than on
$Y$ itself. 
A nonsingular quartic hypersurface $Y=Y_4\subset \PS^4$ is nonrational \cite{IM}; more precisely, $Y$ is the unique Mori fibre space in its birational equivalence class (up to birational automorphisms), i.e.~ $Y_4$ is \emph{birationally rigid}. This remains true for a quartic hypersurface with ordinary double points: \cite{Me04} shows that if $Y$ has ordinary double points but is factorial, $Y$ remains birationally rigid.  When the factoriality assumption is dropped, many examples of rational quartic hypersurfaces with ordinary double points are known. It is a notoriously difficult question to determine which mildly singular quartic hypersurfaces are rational. The methods of this paper yield a partial answer and a byproduct of Theorem~\ref{thm:1} is :  

\begin{cor}
\label{rat}
Let $Y_4\subset \PS^4$ be a non-factorial quartic hypersurface with no worse than terminal singularities. Assume that $Y$ contains a rational scroll as in Theorem~\ref{thm:1}.  Then $Y$ is rational except possibly when $\rk \Cl Y=2$ and $Y$ is the midpoint of a Sarkisov link of type $15,17,25,29,35$ or $36$ in Table~\ref{table1}.
\end{cor}

There are some rationality criteria for strict Mori fibre spaces \cite{Sh83,Al87, Sh07}, so that partial answers are known when a small $\Q$-factorialisation of $Y$ is a strict Mori fibre space.
Terminal quartic hypersurfaces that contain a plane may in some cases be studied directly.
The main open case is therefore the one addressed by Corollary~\ref{rat}. The behaviour of the $6$ cases left out by Corollary~\ref{rat}  and that of factorial terminal quartic hypersurfaces is unclear. Some of these cases would be settled by the following conjecture.

\begin{con}
\label{con:rig}
A factorial quartic hypersurface $Y_4\subset \PS^4$ (resp.~ a generic complete intersection $Y_{2,3}\subset \PS^5$) with no worse than terminal singularities has a finite number of models as
Mori fibre spaces, i.e.~ the pliability of  $Y$ is finite.
\end{con}

Whereas nodal quartic hypersurfaces are birationally rigid, \cite{CM04} constructs an example of a ``bi-rigid'' terminal factorial quartic hypersurface $Y_4\subset \PS^4$.  A nonsingular general complete intersection $Y_{2,3}$ of a quadric and a cubic in $\PS^3$ is birationally rigid \cite{IP96}.However, \cite{ChGr} constructs a small deformation of a (factorial) rigid $Y_{2,3}\subset \PS^5$ with one ordinary double point to a bi-rigid complete intersection of the same type. This example relies on an appropriate deformation of a Sarkisov link between two complete intersections $Y_{2,3}$ (compare with case $35$ in Table~\ref{table1}). Hence, the notion of finite pliability-- rather than that of birational rigidity-- is the one that might behave well in (suitable) families.

Let $Y$ be a terminal quartic hypersurface and $X\to Y$ a small factorialisation. Assume that $X$ is not a strict Mori fibre space. Assuming Conjecture~\ref{con:rig}, $Y$ has finite pliability when $Y$ is factorial or has $\rk \Cl Y=2$ and $Y$ is one of cases $35$ or $36$; $Y$ is rational in all other cases except possibly when $Y$ has $\rk \Cl Y= 2$ and $Y$ is one of cases $15,17,25$ or $29$. In particular, the question of rationality or of finite pliability of $Y$ would be of a topological nature and would be determined by $\Cl Y$.

\subsection*{Outline}
I sketch the proofs of Theorem~\ref{thm:1} and of Corollary~\ref{rat} and present an outline of this paper.

In Section~\ref{weak-star}, I recall the definition of weak-star Fano $3$-folds introduced in \cite{Kal07b}.
If $Y$ is a terminal Gorenstein Fano $3$-fold that does not contain a plane, a small factorialisation $X \to Y$
is a weak-star Fano $3$-fold. The category of weak-star Fano $3$-folds is preserved by the birational operations
of the MMP.  If $X$ is weak-star Fano, then each birational step of the MMP on $X$ is either a flop or an
extremal  divisorial contraction for which the geometric description of Cutkosky-Mori \cite{Cut88} holds.
The end product of the MMP on $X$ is well understood. This approach then yields a complete description
of $\Cl Y$: $\Cl Y/\Pic Y$ is generated by the proper transforms of the exceptional divisors of the divisorial
contractions of the MMP on $X$.

If $X$ has Picard rank $2$ and if $\phi \colon X \to X'$ is a divisorial contration, then $\phi$ is one side of
a Sarkisov link with centre along $Y$. A \emph{$2$-ray game} on $X$ as in \cite{Tak89} determines a
(finite number of) possibilities for the contraction $\phi$; by construction, $\Cl Y$ is then generated by
$\mathcal{O}_Y(1)$ and by the image of $\Exc \phi$ on $Y$.

In the general case, if $\phi \colon X\to X'$ is a divisorial extremal contraction with $\Exc \phi=E$, there is
an extremal contraction $\varphi \colon Z \to Z_1$ where $Z$ is a Picard rank $2$ small modification of $Y$ that
sits under $X$ and such that $\Exc \varphi$ is the image of $E$ on $Z$. Since $Z$ is not factorial, there is a
priori no ``sensible'' Sarkisov link with centre along $Y$ involving $Z$. The proof will rely on
exhibiting a natural link. 

In Section~\ref{weak-star}, I show that the explicit geometric description of extremal
divisorial contractions of \cite{Cut88} holds on non-factorial terminal Gorenstein $3$-folds so long as the exceptional divisor is Cartier.   

Section~\ref{deformation} recalls results on the deformation theory of (small modifications of) terminal Gorenstein Fano $3$-folds and on the deformation of extremal contractions. Following the above notation, there
is a $1$-parameter proper flat deformation $Z\hookrightarrow\mathcal{Z}$, where
$\mathcal{Z}_{\eta}$ is a nonsingular Picard rank $2$ small modification of a terminal Gorenstein Fano $3$-fold
$\mathcal{Y}_{\eta}$ that is a $1$-parameter proper flat deformation of $Y$. The extremal contraction $\varphi$
deforms to an extremal contraction on $\mathcal{Z}_{\eta}$ that is one side of a Sarkisov link with centre along
$\mathcal{Y}_{\eta}$. Each possible Sarkisov link with centre along $\mathcal{Y}_{\eta}$ obtained by the $2$-ray
game on $\mathcal{Z}_{\eta}$ can then be specialised to the central fibre. The specialisation to the central fibre is a Sarkisov link with centre
along $Y$, one side of which is $\varphi$. The divisor class group $\Cl \mathcal{Y}_{\eta}$ is isomorphic to a rank $2$
sublattice of $\Cl Y$.

Section~\ref{motivation} presents the systems of Diophantine equations used in the $2$-ray game.
Roughly, to each extremal contraction one associates numerical constraints on some intersection numbers in
cohomology. If a Sarkisov link involves two extremal contractions $\varphi$ on $Z$ and $\alpha$ on
$\widetilde{Z}$, since $Z$ and $\widetilde{Z}$ are connected by a flop, the constraints on each side
give
rise to systems of Diophantine equations. All possible Sarkisov links with centre along $Y$ are solutions of
these systems. The solutions to all systems associated to Sarkisov links with centre along a terminal Gorenstein
Fano $3$-fold with Picard rank $1$ are listed in Section~\ref{tables}. 
Section~\ref{motivation} classifies terminal quartic hypersurfaces according to their divisor class group. The
case when there is no extremal divisorial contraction on $X$, where the previous arguments do not apply, is
treated separately. 
When $Y$ does not contain a plane, a consequence of the explicit xstudy of the MMP on a small factorialisation $X \to Y$ is
that the bound on the rank of $\Cl Y$ stated in \cite{Kal07b} is not optimal (see Remark~\ref{bound}).

Section~\ref{rationality} first states some easy consequences of the previous explicit study on rationality of
non-factorial Fano $3$-folds. I conjecture that a terminal quartic hypersurface $Y_4\subset \PS^4$ either has
finite pliability or is rational; and that this is determined by $\Cl Y$.  I show that most non-factorial terminal quartic hypersurfaces that do not contain a plane are rational. I then study explicitly quartic hypersurfaces that contain a plane and state some partial results.  
Last, Section~\ref{examples} gives some examples of non-factorial terminal Gorenstein Fano $3$-folds and gathers some observations and remarks.

\subsection*{Notations and conventions}
All varieties considered in this paper are normal, projective and
defined over
$\C$. Let $Y$ be a terminal Gorenstein Fano $3$-fold, $A_Y=-K_Y$ denotes the anticanonical divisor of $Y$.
The \emph{Fano index} of $Y$ is the maximal integer such that $A_Y=i(Y) H$ with $H$ Cartier.
As I only consider Fano $3$-folds
with terminal Gorenstein singularities in this paper, the term index always stands for Fano index.  
The \emph{degree} of $Y$ is $H^3$ and the \emph{genus} of $Y$ is $g(Y)=h^0(X,A_Y)-2$. 
I denote by $Y_{2g-2}$ for $2 \leq g \leq 10$ or $g=12$ (resp.~
$V_d$ for $1\leq d \leq 5$) terminal Gorenstein Fano $3$-folds of Picard rank $1$, index $1$ (resp.~ $2$) and
genus $g$ (resp.~ degree $d$).
Finally, $\F_m= \PS(\mathcal{O}_{\PS^1}\oplus \mathcal{O}_{\PS^1}(-m))$ denotes the
$m$th Segre-Hirzebruch surface.

\subsection*{Acknowledgements}
I would like to thank Alessio Corti, Vladimir Lazi\'c, Miles Reid, Burt Totaro and Chenyang Xu for many useful conversations and helpful comments. I would also like to thank Nick Shepherd-Barron for suggesting applications of some earlier work. Part of this project was completed at the Mathematical Sciences Research Institute, Berkeley. This research was partially supported by Trinity Hall, Cambridge.

\section{Birational Geometry of weak Fano $3$-folds}
\label{weak-star}
In this section, I recall the definition of weak-star Fano $3$-folds and some of their properties.
Most terminal Gorenstein Fano $3$-folds have a weak-star small factorialisation $X\to Y$.
If $X$ is weak-star Fano, the MMP on $X$ is well behaved, i.e.~ there is an explicit geometric description of each step, it terminates  and its end product is either a terminal factorial
Fano $3$-fold or a simple Mori fibre space.
Since $\Cl Y \simeq \Cl X \simeq \Pic X$, the MMP on $X$ yields much information on the divisor class group of
$Y$. I also gather some easy results on elementary contractions on small modifications of terminal
Gorenstein Fano $3$-folds: these will be used in the following Sections.
 
\subsection{Weak-star Fano $3$-folds}

\begin{dfn}\label{dfn:1}
\mbox{}\begin{enumerate}
\item[1.]
A $3$-fold $Y$ with terminal Gorenstein singularities is \emph{Fano} if its anticanonical divisor
$A_Y={-}K_Y$ is ample.
\item[2.]
A $3$-fold $X$ with terminal Gorenstein singularities is \emph{weak Fano} if $A_X$ is
  nef and big.  
  \item[3.]The morphism $X \to Y$ defined by $\vert {-}nK_X \vert$ for $n>\!\!>0$ is the
\emph{(pluri-)anticanonical map} of $X$, $R=R(X, A)$ is the \emph{anticanonical ring} of $X$ and $Y= \Proj R$
is the \emph{anticanonical model} of $X$.
\item[4.] A weak Fano $3$-fold $X$ is a \emph{weak-star Fano} if, in addition:
  \begin{enumerate}
  \item[(i)] $A_X$ is ample outside of a finite set of curves, so that $h\colon X \to Y$ is a small modification,
  \item[(ii)]  $X$ is factorial, and in particular $X$ is Gorenstein,
  \item[(iii)] $X$ is inductively Gorenstein, that is $(A_X)^2\cdot S >1$ for every irreducible divisor $S$
on $X$,
  \item[(iv)] $\vert A_X \vert$ is basepoint free, so that $\varphi_{\vert A \vert}$ is generically finite. 
 \end{enumerate}
\end{enumerate}
\end{dfn}
\begin{rem}\label{rem:1} Let $Y$ be a terminal Gorenstein Fano $3$-fold and $X$ a small factorialisation of $Y$ as in
\cite{Kaw88}.  \cite[Lemma 2.3]{Kal07b} shows that when $Y$ has Picard rank $1$ and genus $g\geq 3$,
$X$ is weak-star unless $Y$ contains a plane $\PS^2$ with ${A_Y}_{\vert \PS^2}=\mathcal{O}_{\PS^2}(1)$.
\end{rem}

\begin{nt}
I call a surface  $S \subset Y_{2g-2}$ a plane (resp.~ a quadric) when the image of $S$ by the anticanonical map
is a plane (resp.~ a quadric) in $\PS^{g+1}$, that is when $(A_Y)^2\cdot S=1$ (resp.~ $2$). 
\end{nt}

\begin{thm}
\label{thm:3}\cite[Theorem 3.2, Lemma 3.3]{Kal07b}
The category of weak-star Fano $3$-folds is preserved by the
birational operations of the MMP.  
More precisely, if $X := X_0$ is a weak-star Fano $3$-fold whose anticanonical
model $Y_0$ has Picard rank $1$, there is a sequence of extremal contractions:
\begin{equation*}
\xymatrix{
X_0 \ar@{-->}[r]^-{\varphi_0} \ar[d]& X_1 \ar@{-->}[r]^-{\varphi_1}\ar[d] & \cdots
& X_{n-1}\ar@{-->}[r]^-{\varphi_{n-1}} \ar[d]& X_n \ar[d] \\
Y_0 & Y_1 & \cdots & Y_{n-1} & Y_n
}
\end{equation*}
where for each $i$, $X_i$ is a weak-star Fano $3$-fold, $Y_i$ is
its anticanonical model, and each
$\varphi_i$ is either a divisorial contraction or a flop. The Picard rank of
$Y_i$, $\rho(Y_i)$, is equal to $1$ for all $i$. The final $3$-fold
$X_n$ is either a Fano $3$-fold with $\rho(X_n)=1$ or a strict  Mori fibre space.
In that latter case, $X_n$ is a del Pezzo fibration over $\PS^1$ or a conic bundle over $\PS^2, \F_0$ or
$\F_2$ and $\rho(X_n)=2$ or $3$.
\end{thm}

\subsection{Elementary contractions of terminal Gorenstein $3$-folds}

Let $h\colon Z\to Y$ be a small modification of a terminal Gorenstein Fano $3$-fold $Y$ with $\rho(Y)=1$ and
$g\geq 3$. 
Suppose that $\varphi \colon Z \to Z_1$ is an extremal contraction such that $E=\Exc \varphi$ is Cartier.
If $g \colon X \to Z$ is a small factorialisation,  and if $\widetilde{E}=g^{\ast} E$, there is an extremal
contraction $\phi \colon X \to X_1$ that makes the diagram 
\begin{eqnarray}
\label{eq:40}
\xymatrix{\widetilde{E} \subset X \ar[r]^{\phi} \ar[d]_g & X_1\ar[d]^{g_1}\\
E \subset Z \ar[d]_h \ar[r]^{\varphi} & Z_1\\
\overline{E} \subset Y &
}
\end{eqnarray}
commutative, where $\widetilde{E}=\Exc \phi$ and $g_1$ is an isomorphism in codimension $1$
(See the proof of \cite[Lemma 3.3]{Kal07b} for details).
\begin{rem}
Note that 
$\overline{E}= h(E)$ is not Cartier: since $\rho(Y)=1$, $E$ would be ample if it were Cartier.     
\end{rem}
Cutkosky extended Mori's geometric description of extremal contractions to terminal Gorenstein factorial
$3$-folds \cite[Theorems 4 and 5]{Cut88}. The next Lemma is an easy generalisation of his results to divisorial
contractions with Cartier exceptional divisor on terminal Gorenstein Fano $3$-folds that are not necessarily
factorial. 
\begin{lem}
 \label{lem:10}\cite[Lemma 3.1]{Kal07b}
 Let $Z$ be a small modification of a terminal Gorenstein Fano $3$-fold $Y$. Assume that $\vert A_Z \vert$ is
basepoint free.
Denote by $\varphi \colon Z \to Z'$ an extremal
divisorial contraction with centre a curve $\Gamma$ and assume that $\Exc \varphi=E$ is a Cartier divisor. 

Then $\Gamma \subset Z'$  is locally a complete intersection and has planar singularities. The contraction
$\varphi$ is locally the blow up of the ideal sheaf $\mathcal{I}_{\Gamma}$. In addition, the following relations
hold :
\begin{align} 
A_Z^3&=(A_{Z'})^3-2(A_{Z'})\cdot \Gamma-2+2p_a (\Gamma) \\
A_Z^2 \cdot E&=
A_{Z'}\cdot \Gamma+2-2 p_a(\Gamma)\\
A_{Z} \cdot E^{2}&=-2+2p_a(\Gamma) \\
E^{3}&=-(A_{Z'})\cdot \Gamma +2 -2 p_a(\Gamma)
 \end{align}
 \end{lem}
\begin{lem}
 \label{lem:7}\mbox{}
Assume that $\varphi \colon Z \to Z_1$ contracts $E$ to a point $P$, then one of the
 following holds: 
\begin{enumerate}   
\item[E$2$:] $(E,\mathcal{O}_{E}(E))\simeq ( \PS^2,
  \mathcal{O}_{\PS^2}(-1))$ and $P$ is nonsingular.
\item[E$3$:]  $(E,\mathcal{O}_{E}(E))\simeq (\PS^1 \times \PS^1,
  \mathcal{O}_{\PS^1 \times \PS^1}(-1,-1))$ and $P$ is an ordinary double point. 
\item[E$4$:] $(E, \mathcal{O}_{E}(-E))\simeq (Q, \mathcal{O}_{Q}(-1))$, where $Q\subset \PS^3$ is an irreducible reduced
  singular quadric surface, and $P$ is a cA$_{n-1}$ point. 
\item[E$5$:] $(E,\mathcal{O}_{E}(E))\simeq ( \PS^2,
  \mathcal{O}_{\PS^2}(-1))$, and $P$ is a point of Gorenstein
  index $2$.
\end{enumerate}
\end{lem}
\begin{proof}
Diagram~\eqref{eq:40} shows that $g$ maps the centre of $\phi$ to the centre of $\varphi$; in particular $\phi$
also contracts a divisor to a point unless $\phi$ has centre along a curve $C$ such that $A_{X_1} \cdot C=0$. In
this case, by Lemma~\ref{lem:10}, $\widetilde{E}\simeq \F_2$ or $\widetilde{E}\simeq \PS^1 \times \PS^1$ and
$\varphi$ is of type E$3$ or E$4$.

I now assume that the centre of $\phi$ is a point. The divisor $E\subset Z$ is Cartier by assumption and
$A_E=(A_Z-E)_{\vert E}$ is ample: $E$ is a Gorenstein, possibly nonnormal, del Pezzo surface. 

The birational morphism $g_{\vert E}\colon \widetilde{E} \to E$ induced by $g$
is an isomorphism outside a finite set of curves. Since $g$ preserves the anticanonical degree of $E$,
Cutkosky's classification \cite{Cut88} shows that this degree is $1,2$ or $4$ and that the normalisation of $E$
is a plane or a quadric.  
Since $E$ is Cartier and $Z_1$ is Cohen Macaulay, the Serre criterion shows that $E$ is nonnormal if and only if
it is not regular in codimension $1$. As the centre of $g_{\vert E}$ is at worst a finite number of points,
$E$ is normal and $E \simeq \widetilde{E}$; the result follows from \cite{Cut88}.
\end{proof}
%\begin{rem}
%An alternative proof of this lemma would be that $g$ is an isomorphism near $\widetilde{E}$ %because $g(\widetilde{E})$ is
%Cartier. 
%\end{rem}
\begin{lem}
 \label{lem:8}
Let $Y_4\subset \PS^4$ be a non-factorial terminal quartic hypersurface
that does not contain a plane. Let $Z\to Y$ be a small
modification such that $\rho(Z/Y)=1$. Assume that $\varphi
\colon Z \to Z_1$ is an
extremal contraction such that $E=\Exc \varphi$ is Cartier and that $E$ is
mapped to a curve $\Gamma$.
Then $Z_1$ is a terminal Gorenstein Fano $3$-fold with $\rho(Z_1)=1$
and the following relations hold: 
\begin{enumerate}
\item[1.] If $i(Z_1)=1$ and $A_{Z_1}^3=2g_1-2$, then $\deg(\Gamma)=g_1-4+p_a(\Gamma)$ and $p_a(\Gamma) \leq g_1-1$,
\item[2.] If $i(Z_1)=2$ and
  $A_{Z_1}^3= 8d$, then
  $2\deg(\Gamma)= 4d-3+p_a(\Gamma)$ and $p_a(\Gamma)=2k+1$, for some $0\leq k\leq 2d-1 $ 
\item[3.] If $Z_1$ is a quadric in $\PS^4$, then $3
  \deg(\Gamma)= 24+p_a(\Gamma)$ and $p_a(\Gamma)= 3k$, for some $0 \leq  k \leq 9$,
\item[4.] If $Z_1= \PS^3$, then $4 \deg(\Gamma)=29+p_a(\Gamma) $
  and $p_a(\Gamma)= 4k-1$, for some $0 \leq k \leq 7$.
\end{enumerate}
\end{lem}
\begin{rem}
\label{rem:35} Note that the bound obtained on the genus of $\Gamma$ is
sharper than the Castelnuovo bound when $\Gamma$ is a nonsingular curve. 
\end{rem}
\begin{proof}
Since $\vert A_Z\vert = \vert \varphi^{\ast}A_{Z_1}-E\vert$ is basepoint free, $\Gamma$ is a scheme theoretic
intersection of members of $\vert A_{Z_1} \vert$, and hence $A_{Z_1}\cdot \Gamma\leq A_{Z_1}^3$. The Lemma
then follows from standard manipulation of the relations of Lemma~\ref{lem:10}.
\end{proof}
\section{Deformation theory}
\label{deformation}

This Section first recalls results on the deformation theory of terminal Gorenstein Fano $3$-folds and
of their small modifications. I then state an easy extension of results of \cite{Mo82, KM92} on
deformations of extremal contractions.  
As is explained in the Introduction, if $X \to Y$ is a small
factorialisation of a terminal Gorenstein Fano $3$-fold and if $\rho(X)=2$, a \emph{$2$-ray game} on $X$ as
in \cite{Tak89} determines all possible Sarkisov links with centre along $Y$ and hence
all possible $K$-negative extremal contractions $\varphi \colon X \to X'$ and all possible generators of $\Cl Y/\Pic Y$.
In the general case, I show that a similar $2$-ray game can be played on $Z$, a small partial
factorialisation of $Y$ with $\rho(Z)=2$. This procedure is delicate because $Z$ is not factorial. However,
$Z$ can be smoothed and the $2$-ray game on the generic fibre yields Sarkisov links that specialise to
appropriate Sarkisov links involving $Z$ with centre along $Y$. All
possible generators of $\Cl Y/\Pic Y$ arise in that way.  

\subsection{Deformation Theory of weak Fano $3$-folds}
\begin{dfn}
\label{kura}
Let $X$ be a projective variety. 
The \emph{Kuranishi space} $\Def(X)$ of $X$ is the semi-universal
space of flat deformations of $X$. When the functor of flat deformations of $X$ is pro-representable,
the Kuranishi family $\mathcal{X}$ is the universal deformation object.
\end{dfn}
\begin{thm}\cite{Nam97a}
  \label{thm:7}
Let $X$ be a small modification of a terminal Gorenstein Fano $3$-fold.  
There is a $1$-parameter flat deformation of $X$ 
\[
\xymatrix{
X\ar[r] \ar[d]& \mathcal{X} \ar[d]\\
\{0\} \ar[r]& \Delta
}
\]
such that the generic fibre $\mathcal{X}_{\eta}$ is a nonsingular small modification of a terminal Gorenstein
Fano $3$-fold. The Picard ranks, the anticanonical
degrees and the indices of $X$ and $\mathcal{X}_{\eta}$ are equal.
\end{thm}
Let $f\colon X \to Y$ be a small modification of a terminal Gorenstein Fano $3$-fold.
Let $E$ be a Cartier divisor such that $\overline{E}=f(E)$ is not
Cartier, and denote by $Z$ the symbolic blow up of
$\overline{E}$ on $Y$, i.e.~ $Z= \Proj_Y \bigoplus_{n \geq 0}\mathcal{O}_Y(n\overline{E})$. Then $f$ naturally decomposes as:
\[
f \colon X  \stackrel{h}\to Z \stackrel{g}\to Y.
\]
Note that if $\rho(X/Y)=1$, $h$ is the identity and $Z=X$. I recall some
results that relate the deformations of $X, Z$ and $Y$.
\begin{pro}\cite[11.4,11.10]{KM92}
\label{pro:1}
Let $X$ be a normal projective $3$-fold and $f \colon X \to Y$ a proper map with connected fibres such that
$R^1f_{\ast}\mathcal{O}_X=0$. 
\begin{enumerate}
\item[1.]
There are natural morphisms $F$ and $\mathcal{F}$ that make the diagram
\[
\xymatrix{
\mathcal {X} \ar[r]^{\mathcal{F}}  \ar[d] & \mathcal{Y} \ar[d]\\
\Def(X) \ar[r]^{F} & \Def(Y)
}
\]
commutative. In addition, $\mathcal{F}$ restricts to $f$ on $X$.
\item[2.]
Assume further that $X$ has terminal Gorenstein singularities and that
 $f$ contracts a curve $C \subset X$ with $A_{X}$-trivial
components to a point $\{Q\} \in Y$.
Let $X_S \to S$ be a flat deformation of $X$ over the germ of a
complex space $0 \in S$.
Then, $f$ extends to a contraction $F_S \colon X_S \to Y_S$,and the
flop $F_S^+ \colon X_S^+ \to Y_S $ exists and commutes with any base change.
\end{enumerate}
\end{pro}
\begin{thm}\cite[12.7.3-12.7.4]{KM92}
\label{thm:2}
Let $f\colon X \to Y$ be a small factorialisation of a terminal Gorenstein $3$-fold $Y$.
Then, $F\colon \Def(X)\to \Def(Y)$ is finite and $\im [\Def(X) \to \Def(Y)]$
is closed and independent of the choice of $f$.
\end{thm}
\begin{rem}
\label{stratification}
By Proposition~\ref{pro:1}, there are maps
$\mathcal{G}$ and $\mathcal{H}$ that
restrict to $g$ and $h$ on the central fibre and that make the diagram 
\[
\xymatrix{
\mathcal{X}\ar[r]^{\mathcal{H}}\ar[d] &
\mathcal{Z}\ar[r]^{\mathcal{G}}\ar[d]& \mathcal{Y}\ar[d] \\
\Def(X)\ar[r]^{H} & \Def(Z)\ar[r]^{G} & \Def(Y)
}
\]
commutative. 
The Kuranishi space of $X$ thereby acquires a natural stratification by sublattices of $\Cl Y$;
by Theorem~\ref{thm:2}, there is an inclusion of
closed subspaces \[
\Def(X) \subset \Def(Z) \subset \Def(Y).
\] 
As the Picard rank is constant in any $1$-parameter deformation of
$Z$,  $\Def(Z) \subset \Def(Y)$ corresponds to the
locus of the Kuranishi space where the algebraic cycle representing $E$ is preserved.
Further, these inclusions are strict because a smoothing of $Y$ does
not sit under any $1$-parameter flat deformation of $Z$. 
\end{rem}
\subsection{Deformation of extremal rays}

For future reference, I state a mild generalisation of the results on
deformation of extremal rays in \cite{Mo82, KM92}.
\begin{thm}[Extension of extremal contractions] \mbox{}
\label{thm:5}
Let $Z \to Y$ be a small modification of a terminal Gorenstein Fano $3$-fold $Y$.
Consider a projective flat deformation $\mathcal{Z} \to S$ 
of $Z$, where $S$ is a smooth affine complex curve with closed point $\{0\}$ and generic point
$\eta$. 
Let $\varphi \colon Z \to Z_1$ be the contraction of an extremal ray $R
\subset Z$ and assume that if $\Exc \varphi$ is a
divisor, it is Cartier.
The contraction $\varphi$ extends to an $S$-morphism $f \colon
\mathcal{Z} \stackrel{\Phi} \to \mathcal{Z}_1$, where $\mathcal{Z}_1\to S$
is a projective $1$-parameter flat deformation of $Z_1$, and    %CHECK THAT IT IS ACTUALLY PROJECTIVE!!--yes ok, relative bpf-ness
\begin{enumerate}
 \item[1.] $\Phi_{\eta}$ is the contraction of an extremal ray,
 \item[2.] If $\varphi= \Phi_{0}$ contracts a subset
 of $\codim \geq 2$ (resp.~ a divisor, resp.~ is a fibre space of generic
 relative dimension $k$),  so does $\Phi_\eta$, 
\item[3.] If $\Exc \varphi$ is a Cartier divisor, in the notation of Lemma~\ref{lem:7},
 either $\Phi_{\eta}$ and $\varphi$ are of the same type, or $\Phi_{\eta}$
 and $\varphi$ are of types E$3$ and E$4$.
\end{enumerate}
\end{thm}
\begin{proof}
The assumption that $E$ is Cartier ensures that the proof of
\cite[Theorem 3.47]{Mo82} can be extended to this case.
See \cite{Kal07a} for a complete proof.
\end{proof}

\begin{thm}[The $2$-ray game]
\label{thm:2ray}
Let $Z \to Y$ be a small modification of a terminal Gorenstein Fano $3$-fold $Y$ with $\rho(Y)=1$ and
$\rho(Z/Y)=1$.
Assume that $\varphi \colon Z \to Z_1$ is a divisorial contraction and that $E=\Exc \varphi$ is Cartier.
There is  a diagram:
\begin{eqnarray}
\label{eq:41}
\xymatrix{ \quad & Z \ar[dl]_{\varphi} \ar[dr]^{g}
\ar@{<-->}[rr]^{\Phi} &\quad
& \widetilde{Z} \ar[dr]^{\alpha} \ar[dl]_{\tilde{g}} &\quad\\
Z_1  &\quad & Y & \quad & \widetilde{Z_1}}
\end{eqnarray}
where:
\begin{enumerate}
\item[1.] $Z$ and $\widetilde{Z}$ are small modifications of $Y$ with Picard rank
  $2$,
\item[3.] $\Phi$ is a composition of flops that is not an isomorphism,
\item[4.] $\alpha$ is a $K$-negative extremal contraction,
\item[5.] $Z_1$ (resp~$\widetilde{\mathcal{Z}}_1$) is one of:
  \begin{enumerate}
  \item[(i)] a terminal Gorenstein Fano $3$-fold with Picard rank $1$ if $\varphi$ (resp.~$\alpha$) is
birational,
  \item[(ii)] $\PS^2$ if $\varphi$ (resp.~$\alpha$) is a conic bundle,
  \item[(iii)] $\PS^1$ if $\varphi$ (resp.~$\alpha$) is a del Pezzo fibration.
  \end{enumerate}
\end{enumerate}
If $\alpha$ is birational, then $\Exc \alpha$ is Cartier.
\end{thm}
\begin{rem}
I want to stress that since $Z$ is not factorial, such a diagram does not automatically exist, and in particular,
$\Exc \alpha$ is not necessarily Cartier when it is a divisor. 
\end{rem}
\begin{proof}
By Theorem~\ref{thm:7}, there is a $1$-parameter smoothing $\mathcal{Z}\to \Delta$ of $Z$. For all
$t\in \Delta\smallsetminus \{0\}$, $\mathcal{Z}_t$ is a nonsingular small modification of a terminal Gorenstein
Fano $3$-fold with $\rho(\mathcal{Z}_t)=2$. Let $g_t \colon \mathcal{Z}_t \to \mathcal{Y}_t$ be the
anticanonical map. 
Note that Proposition~\ref{pro:1} ensures that $\mathcal{Y}_t$ is a $1$-parameter flat deformation of $Y$;
in particular $\mathcal{Y}_t$ is a terminal Gorenstein Fano $3$-fold with $\rho(\mathcal{Y}_t)=1$ and
$A_{\mathcal{Y}_t}^3=A_Y^3$, . 

Theorem~\ref{thm:5} shows that there is an extremal contraction $\varphi_t $ of $\mathcal{Z}_t$ that specialises
to $\varphi$ on the central fibre. As
$\mathcal{Z}_t$ is factorial, a $2$-ray game as in \cite{Tak89} yields a diagram:
\[
\xymatrix{ \quad & \mathcal{Z}_t \ar[dl]_{\varphi_t} \ar[dr]^{h_t}
\ar@{<-->}[rr]^{\Phi_t} &\quad
& \widetilde{Z}_t \ar[dr]^{\alpha_t} \ar[dl]_{\widetilde{h}_t} &\quad\\
\mathcal{Z}_{1,t}  &\quad & \mathcal{Y}_t & \quad &
\widetilde{\mathcal{Z}_{1,t}},}
\]
where 
\begin{enumerate}
\item[1.] $\mathcal{Z}_t$ and $\widetilde{\mathcal{Z}_t}$ are
  nonsingular small modifications of $\mathcal{Y}_t$ with Picard rank
  $2$,
\item[2.] $\Phi_t$ is a composition of flops that is not an isomorphism,
\item[3.] $\alpha_t$ is a $K$-negative extremal contraction,
\item[4.] $\mathcal{Z}_{1,t}$ (resp.~ $\widetilde{\mathcal{Z}_{1,t}}$) is one of:
  \begin{enumerate}
  \item[(i)] a terminal Gorenstein Fano $3$-fold with Picard rank $1$ if $\varphi_t$ (resp.~ $\alpha_t$) is
birational,
  \item[(ii)] $\PS^2$ if $\varphi_t$ (resp.~ $\alpha_t$) is a conic bundle,
  \item[(iii)] $\PS^1$ if $\varphi_t$ (resp.~ $\alpha_t$) is a del Pezzo fibration.
  \end{enumerate}
\end{enumerate}
The theorem then follows from Lemma~\ref{lem:spe}.
\end{proof}
\begin{lem}[Specialisation of a $2$-ray game]
\label{lem:spe}
 The elementary Sarkisov link on $\mathcal{Z}_t$, $t \neq 0$, induces
 an elementary Sarkisov link on the central fibre of $\mathcal{Z} \to \Delta$. 
\end{lem}
\begin{proof}
This lemma is standard, see \cite{Kal07a} for a proof; it follows from the more general
\cite[Theorem 4.1]{dFH09}.
\end{proof}

Let $Y$ be a terminal Gorenstein Fano $3$-fold with $\rho(Y)=1$. Assume that $X$, a small factorialisation of
$Y$, is weak-star Fano. Theorem~\ref{thm:3} shows that there is a sequence of contractions:
\begin{equation}
\label{eq:7}
\xymatrix{
X_0 \ar@{-->}[r]^-{\phi_0} \ar[d]& X_1 \ar@{-->}[r]^-{\phi_1}\ar[d] & \cdots
& X_{n-1}\ar@{-->}[r]^-{\phi_{n-1}} \ar[d]& X_n \ar[d] \\
Y_0 & Y_1 & \cdots & Y_{n-1} & Y_n
}
\end{equation}
I assume that at least one of the contractions $\varphi_i$ is divisorial. 
Then, for a suitable small factorialisation $X_0$, $\varphi_{0}= \varphi$ is divisorial.
Let $\widetilde{E}= \Exc \phi$ and $Z_0$ be a small modification of $Y$ such that $X_0\to Y_0$ factors through
$Z_0$, $\rho(Z_0/Y_0)=1$ and such that $E$, the image of $\widetilde{E}$ on $Z_0$, is Cartier. Then there is an
extremal contraction $\varphi \colon Z_0\to Z_1=Y_1$, such that the diagram
\begin{eqnarray}
\label{eq:1}
\xymatrix{\widetilde{E} \subset X \ar[r]^{\phi} \ar[d]_g & X_1\ar[d]^{g_1}\\
E \subset Z \ar[d]_h \ar[r]^{\varphi} & Z_1\\
\overline{E} \subset Y &
}
\end{eqnarray}
commutes. 
Theorem~\ref{thm:2ray} shows that $Z_0,Y_0$ and $Z_1$ fit in an elementary Sarkisov link as in \eqref{eq:41}.
To each such elementary link, one can associate systems of Diophantine equations  that reflect the numerical
constraints imposed by contractions of extremal rays on intersection of classes in cohomology. 
These constraints can be made explicit when $\Exc \varphi$ (and $\Exc \alpha$, if it is a divisor) is Cartier,
as is explained in Lemma~\ref{lem:10} and in Section~\ref{motivation}.

This procedure can be carried out at each divisorial step of the MMP on $X_0$.

\section{A geometric motivation of non-factoriality}
\label{motivation}
In this section, I write down explicitly systems of Diophantine equations associated to elementary Sarkisov
links as in \eqref{eq:41}. As Lemma~\ref{lem:10} shows, one can associate to each extremal contraction numerical
constraints. These systems of Diophantine equations reflect the relationships between the constraints associated
to the extremal contractions $\varphi$ and $\alpha$ on both sides of the link.

I then list all possible divisorial contractions that can occur when running the MMP on weak-star Fano $3$-folds
$X$ whose anticanonical model $Y$ have $\rho(Y)=1$. Not all links listed in this Section are geometrically realizable, I call them \emph{numerical links} in order to stress this fact. 

\subsection{Systems of Diophantine equations associated to elementary links}
\label{2raygame}
In this section, I use the notation set in \eqref{eq:41}. Let $\widetilde{E}$ be the proper transform of
$E=\Exc\varphi$ on $\widetilde{Z}$, $\widetilde{E}$ is a Cartier divisor. 
In what follows, I assume that $i(\widetilde{Z})= i(Z)=i(Y)=1$. This is a convenience, Remark~\ref{rem:hif}
explains how to recover the general case. 
By construction, $H$ and $\widetilde{E}$
are generators of $\Pic \widetilde{Z}$.
Let $g$ denote the genus of $Y,Z$ and $\widetilde{Z}$.

Since $\Phi$ is a sequence of flops, 
\begin{eqnarray}
  \label{eq:17}
  A_Z^2 {\cdot} E & = A_{\widetilde{Z}}^2 {\cdot} \widetilde{E} \nonumber\\
 A_Z {\cdot} E^{2} & = A_{\widetilde{Z}}{\cdot} \widetilde{E}^{2}\\
\widetilde{E}^{3} & =E^{3}-e.\nonumber
\end{eqnarray}
\begin{lem}\cite{T02}
 \label{lem:4}
The correction term $e$ in \eqref{eq:17} is a strictly positive integer.
\end{lem}
\begin{proof}
This is standard, I include the argument for clarity of exposition.
As $E$ is Cartier and $\varphi$-negative, for any effective curve $\gamma \in \Exc \Phi$, 
$E\cdot\gamma$ is strictly positive. 
Since $\widetilde{E}$ is also Cartier, $e$ is an integer.

Consider a common resolution of $Z$ and $\widetilde{Z}$:
\[
\xymatrix{
\quad & W\ar[dr]^{q}\ar[dl]_{p}& \quad \\
Z \ar@{<-->}[rr]^{\Phi} &\quad & \widetilde{Z},
}
\]
Since $Z$ and $\widetilde{Z}$ are terminal, the Negativity Lemma shows that
every $p$-exceptional divisor is also $q$-exceptional and that $p^{\ast}A_Z= q^{\ast}A_{\widetilde{Z}}$.

Then,
\[p_{\ast}^{-1}E= p^{\ast} E-R=
q^{\ast}(\widetilde{E})-R',\]
where $R$ and $R'$ are effective exceptional divisors for $p$ and $q$.
In particular:
\[
-p^{\ast}(E)=-q^{\ast}(\widetilde{E})+R'-R.
\]  
By the construction of the $E$-flop, $-q^{\ast}(\widetilde{E})$ is
$p$-nef. The Negativity
Lemma shows that $R'-R$ is strictly effective because $\Phi$ is not
an isomorphism, and its pushforward $p_{\ast}(R'-R)$ is effective.
Hence,
$-p_{\ast}(R'-R)^2$ is a non-zero
effective $1$-cycle contained in the indeterminacy locus of $\Phi$, and $e=-p_{\ast}(R'-R)\cdot E>0$. 
\end{proof}

I now write down numerical constraints associated to the extremal
contraction $\alpha$; this is similar to what is done in Lemma~\ref{lem:10} . These constraints and \eqref{eq:17} yield the systems of Diophantine equations that underlie the
$2$-ray game.

\subsubsection{$\alpha$ is divisorial}
Let $D$ be the exceptional divisor of $\alpha$ and $C$ its centre. Since $D$ is Cartier, there are integers $x, y$ such that:
\begin{equation}
  \label{eq:21}
D=x A_{\widetilde{Z}} - y \widetilde{E}.  
\end{equation}
If $\alpha$ is of type  E$1$,
\begin{equation}
  \label{eq:20}
A_{\widetilde{Z}}= \alpha^{\ast}(A_{\widetilde{Z}_1})-D.  
\end{equation}
Since $\widetilde{Z}_1$ is Gorenstein, $\alpha(\widetilde{E})$ is Cartier because it is $\Q$-Cartier;
\eqref{eq:21} and \eqref{eq:20} show that
$y$ divides $x+1$. Note that $y$ is the index of $\widetilde{Z}_1$ and define $k$ by $x+1=yk$.
By Lemma~ \ref{lem:10},
\[
\left \{
\begin{array}{c}
(A_{\widetilde{Z}}+D)^3=(A_{\widetilde{Z}}+D)^2 A_{\widetilde{Z}}=(A_{\widetilde{Z}_1})^3\\
(A_{\widetilde{Z}}+D)^2 D=0\\
(A_{\widetilde{Z}}+D)DA_{\widetilde{Z}}=A_{\widetilde{Z}_1}\cdot C=i(\widetilde{Z}_1) \deg(C)\\
A_{\widetilde{Z}}D^2=2p_a(C)-2
\end{array}
\right .
\]
These relations and \eqref{eq:17} yield the system of equations associated to the
configuration $(\varphi, \alpha)$:
\[
\left \{
\begin{array}{c}
y^2[A_Z^3k^2-2(A_{Z_1}\cdot
\Gamma+2-2p_a(\Gamma))k+2p_a(\Gamma)-2]\nonumber
\\=A_{\widetilde{Z_1}}^3  \nonumber \\  
A_Z^3k^2(yk-1)+(A_{Z_1}\cdot
\Gamma+2-2p_a(\Gamma))(2k-3k^2y)\nonumber \\+ (2p_a(\Gamma)-2)(3ky-1)+(A_{Z_1}\cdot
\Gamma-2+2p_a(\Gamma)+e)y=0  \nonumber\\
A_Z^3k(yk-1)-(A_{Z_1}\cdot
\Gamma+2-2p_a(\Gamma))(2yk-1)\nonumber \\+ (2p_a(\Gamma)-2)y= 
\deg(C)  \nonumber\\
A_Z^3(yk-1)^2-2(A_{Z_1}\cdot
\Gamma+2-2p_a(\Gamma))y(yk-1)\nonumber \\
+(2p_a(\Gamma)-2)y^2=2p_a(C)-2  \nonumber
\end{array}
\right .
\]

\begin{rem}
Assume that the degree of $Z$ is fixed.
Since $Z_1$ and $\widetilde{Z_1}$ are terminal Gorenstein Fano $3$-folds with Picard rank $1$, there are
finitely many possible values for ${A_{Z_1}}^3$ and ${A_{\widetilde{Z_1}}}^3$. Further, once  ${A_{Z_1}}^3$ and ${A_{\widetilde{Z_1}}}^3$ are fixed, Lemma~\ref{lem:8} shows that there are only finitely many
possibilities for $(p_aC, \deg C)$. As a result, there is a finite number of  Diophantine systems to consider to determine all numerical Sarkisov links $(\varphi, \alpha)$ with centre along $Z$.
\end{rem}

\subsubsection{$\alpha$ is a conic bundle}
Let $L$ be the pull back of
an ample generator of $\Pic\widetilde{ Z_1}$. Since $\rho(\widetilde{Z_1})=2$,
$\widetilde{Z_1}=\PS^2$ \cite[Lemma 3.6]{Kal07b} and $L=\alpha^{\ast}\mathcal{O}_{\PS^2}(1)$.

There are integers $x, y$ such that:
\begin{eqnarray}
  \label{eq:22}
L=x A_{\widetilde{Z}} - y \widetilde{E}.  
\end{eqnarray}
 \begin{cla}
The integers $x$ and $y$ are positive and coprime; $y$ is
equal to $1$ or $2$.  
\end{cla}
This is similar to the argument in \cite{Tak89}. 
Since $E$ is fixed on $Z$,
$\widetilde{E}$ is fixed, and hence $x\geq 0$. 
If $y\leq 0$, $\vert L \vert \supset \vert x
A_{\widetilde{Z}}\vert$, and $L$ is big. This contradicts
$\alpha$ being of fibering type. The integers $x, y$ are
coprime because $A_{\widetilde{Z}}$ and $\widetilde{E}$ form a
$\Z$-basis of $\Pic  \widetilde{Z}$, $L$ is prime and $L$ is not an
integer multiple of either of them.
Denote by $l$ an effective nonsingular curve that is contracted by $\alpha$.
Then $A_{\widetilde{Z}}\cdot l \leq 2$, and since
$x(A_{\widetilde{Z}}\cdot l)=y\widetilde{E}\cdot
l$, the claim follows.

Let $\Delta\sim-\alpha_{\ast}(A_{\widetilde{Z}/\widetilde{Z}_1})^2$ be the discriminant curve of $\alpha$.
\[
\left \{ \begin{array}{c}
L^3=0\\
L^2 {\cdot} A_{\widetilde{Z}}=2 \\
L {\cdot} A_{\widetilde{Z}}^2 =12-\deg(\Delta) 
\end{array} \right.
\]
The system of equations associated to the configuration $(\varphi, \alpha)=(E1, CB)$ writes:
\[
\left \{
\begin{array}{c}
A_Z^3x^3
-3(A_{Z_1}{\cdot} \Gamma+2-2p_a(\Gamma))x^2y
\nonumber \\+3(2p_a(\Gamma)-2)xy^2 
+(A_{Z_1}{\cdot}
\Gamma-2+2p_a(\Gamma)+e)y^3=0\\
A_Z^3x^2-2(A_{Z_1}{\cdot} \Gamma+2-2p_a(\Gamma))xy\nonumber\\+(2p_a(\Gamma)-2)y^2=2\\
A_Z^3x-(A_{Z_1}{\cdot} \Gamma+2-2p_a(\Gamma))y= 12-\deg(\Delta)
\end{array}
\right .
\]

\subsubsection{$\alpha$ is a del Pezzo fibration}
Let $L$ be the pullback of an ample generator of $\Pic \widetilde{Z_1}$.
As $\widetilde{Z_1}=\PS^1$, $L=\alpha^{\ast}\mathcal{O}_{\PS^1}(1)$. Let $d$ be the degree of the generic
fibre. There are integers $x,y$ such that \eqref{eq:22} holds.
\begin{cla}
The integers $x$ and $y$ are positive and coprime; $y$ can only be
equal to $1,2$ or $3$.  
\end{cla}
This is proved as in the conic bundle case.
\[
\left \{
\begin{array}{c}
L^2{\cdot} A_{\widetilde{Z}}=0\\
L^2 {\cdot} \widetilde{E}=0\\
L{\cdot}A_{\widetilde{Z}}^2=d
\end{array}
\right .
\]
The system of equations associated to $(\varphi, \alpha)$ writes:
\[
\left \{
\begin{array}{c}
A_Z^3x^2-2(A_{Z_1}{\cdot}
\Gamma+2-2p_a(\Gamma))xy+\nonumber\\(2p_a(\Gamma)-2)y^2=0 \nonumber
\\
(A_{Z_1}{\cdot} \Gamma+2-2p_a(\Gamma))x^2-2(2p_a(\Gamma)-2)xy \nonumber\\-(A_{Z_1}{\cdot}
\Gamma-2+2p_a(\Gamma)+e)y^2=0  \nonumber\\
A_Z^3x-(A_{Z_1}{\cdot} \Gamma+2-2p_a(\Gamma))y= d  \nonumber
\end{array}
\right .
\]
\begin{rem}
\label{rem:hif}
 \cite{Sh89} shows that when $i(Y)=4$, $Y$ is isomorphic to $\PS^3$.
If $i(Y)=2$,
both $\alpha$ and $\varphi$ are either E$2$ contractions, \'etale conic bundles or quadric bundles.
If $i(Y)=3$, then $\alpha$ and $\varphi$ are $\PS^2$-bundles over $\PS^1$.
The MMP on small modifications of higher index Fano $3$-folds is therefore very simple. 
If $\widetilde{H}$ is such that  $A_{\widetilde{Z}}= i(\widetilde{Z})H$, the systems written above hold
for any index after replacing $A_{\widetilde{Z}}$  by $\widetilde{H}$ in \eqref{eq:21} and \eqref{eq:22}.
\end{rem}
%\begin{table}[p]

\begin{table}
  \caption{Numerical Sarkisov Links, $g=3$}
  \label{table1}
  \begin{tabular}{rclclccc}
\hline
 &$(\varphi, \alpha)$ &$Z_1$& $ \varphi$ & $\widetilde{Z_1}$ & $\alpha$   & $e$ & $R$\\ 
\hline
$1$ & E$1$-E$1$ &$X_{22}$ &$(g,d)= (0,8)$ & $X_{22}$  &
$(h,e)=(0, 8)$  & $268$ & $+$\\
 $2$ & E$1$-E$1$& $X_{22}$ & $(g,d)= (1,9)$ & $V_5$ & 
$(h,e)= (1,9)$  & $171$ & $+$\\ 
 $3$ & E$1$-E$1$& $X_{22} $ &  $(g,d)=(2,10)$ & $X_{22}$ & 
$(h,e)= (2,10)$ & $80$& $+$\\ 
 $4$ & E$1$-CB & $X_{22}$ &   $(g,d)=(2,10)$ & $ \PS^{2}$ &
$\deg \Delta= 4$ &$92$& $+$\\
 $5$& E$1$-E$1$ & $X_{22}$ &  $(g,d)=(3,11)$ & $X_{12}$ & 
$(h,e)=(0,3) $ & $29$& $+$\\ 
 $6$ & E$1$-E$1$& $X_{18}$  & $(g,d)=(0,6)$ & $X_{18}$& 
$(h,e)= (0,6)$ & $144$ & $+$\\ 
 $7$ & E$1$-E$1$& $X_{18}$  & $(g,d)=(1,7)$ & $V_4$& 
$(h,e)=(1,7)$ & $77$ & $+$\\ 
 $8$& E$1$-E$1$ & $X_{18}$  &  $(g,d)=(2,8)$& $X_{18}$ & 
$(h,e)= (2,8)$ & $16$& $+$\\ 
 $\times 9$ & E$1$-CB& $X_{18}$  &  $(g,d)=(2,8)$ & $\PS^{2}$&
$\deg \Delta = 6$  & $26$ & $+$ \\ 
 $10$ & E$1$-E$1$& $X_{16}$ & $(g,d)=(0,5)$ & $Q $ & 
$(h,e)= (3,9)$ & $103$& $+$\\ 
 $11$ & E$1$-E$1$& $X_{16}$ &  $(g,d)=(1,6)$ & $X_{16}$ & 
$(h,e)=(1,6)$  & $42$& $+$\\ 
 $12$ & E$1$-dP& $X_{16}$  &  $(g,d)=(1,6)$ & $\PS^{1}$& $k=6$  & $48$& $+$\\ 
 $13$ & E$1$-E$1$& $X_{16}$  &  $(g,d)=(2,7)$ & $X_{2,2,2}$ & 
$(h,e)= (0,1)$ & $8$ & $+$\\
 $14$ & E$1$-E$1$& $X_{16}$ &  $(g,d)=(2,7)$  & $V_4$& 
$(h,e)= (5,9)$  & $4$& $+$\\ 
 $15$ & E$1$-E$1$& $X_{14}$ &  $(g,d)=(0,4)$ & $X_{14}$ & 
$(h,e)= (0,4)$  & $68$ & ?\\ 
 $16$ & E$1$-E$1$& $X_{14}$  &  $(g,d)=(1,5)$ & $Q$& 
$(h,e)=(9, 11)$ & $24$ & $+$\\ 
 $\bullet17$ & E$1$-E$1$& $X_{14}$  & $(g,d)=(1,5)$& $V_3$ & 
$(h,e)= (1,5)$  & $25$& ?\\ 
 $18$ & E$1$-E$1$& $X_{12}$ & $(g,d)=(0,3)$  & $X_{22}$& 
$(h,e)= (3,11)$  &$29$ & $+$\\
 $\bullet19$ & E$1$-E$1$& $X_{12}$  & $(g,d)=(0,3)$& $\PS^{3}$ & 
$(h,e)=(7, 9)$ & $45$ & $+$\\
 $20$ & E$1$-E$1$& $X_{12}$&  $(g,d)=(1,4)$  & $X_{12}$ & 
$(h,e)= (1,4)$   & $8$& $+$\\ 
 $21$ & E$1$-dP& $X_{12}$  &  $(g,d)=(1,4)$ & $\PS^1$&  $k=4$ & $12$ & $+$\\ 
 $\bullet22$& E$2$-E$2$&  $X_{12}$  &&  $X_{12}$  && $30$ & $+$\\
 $\times 23$& E$2$-E$1$& $X_{12}$&   & $X_{10}$ & 
$(h,e)= (0,2)$   & $29$& $+$\\ 
 $\times 24$& E$2$-E$1$& $X_{12}$&   & $V_{5}$ & 
$(h,e)= (7,12)$   & $24$& $+$\\ 
 $\bullet25$ & E$1$-E$1$& $X_{10}$ &  $(g,d)=(0,2)$ & $X_{10}$ & 
$(h,e)= (0,2)$ & $28$ & ?\\
 $\times 26$ & E$1$-E$1$& $X_{10}$  &  $(g,d)=(0,2)$ & $V_5$& 
$(h,e)= (7,12)$ & $23$& $+$\\ 
 $\times27$ & E$1$-E$2$& $X_{10}$  &  $(g,d)=(0,2)$ &  $X_{12}$&& $29$ & $+$ \\ 
 $\bullet28$ & E$1$-E$1$& $X_{10}$  &  $(g,d)=(1,3)$ &  $\PS^3$& 
$(h,e)= (15,11)$   & $2$ & $+$\\ 
 $\times 29$ & E$1$-E$1$& $X_{10}$ &  $(g,d)=(1,3)$&  $V_2$  & $(h,e)= (1,3)$ & $3$ & ?\\ 
 $\bullet30$ & E$1$-CB& $X_{2,2,2}$  & $(g,d)=(0,1)$ & $\PS^{2}$ & $\deg
\Delta = 7$ & $17$& \\ 
 $\times31$ & E$1$-E$1$& $X_{2,2,2}$ &   $(g,d)=(0,1)$ &$X_{16}$ & 
$(h,e)=(2,7)$  & $8$ & $+$\\ 
$32$ & E$1$-dP& $V_2$  &  $(g,d)=(1,3)$ &  $\PS^1$ & $k=6$  & $48$& $+$\\ 
$33$& E$1$-E$1$ & $V_2$ &  $(g,d)=(1,3)$ & $X_{16}$ &  $(h,e)=(1,6)$  & $42$& $+$\\
 $\bullet34$ & E$1$-E$1$& $V_3$ & $(g,d)=(3,6)$ & $\PS^{3}$ &  $(h,e)=(3,8)$ & $65$ & $+$
\\ 
 $\bullet35$ &E$3$-E$3$& $X_{2,3}$ & $(g,d)=(0,0)$ & $X_{2,3}$ &  $(h,e)=(0,0)$ & $12$ & $?$ \\
 $36$ &E$3$-E$1$& $X_{2,3}$ & $(g,d)=(0,0)$ & $V_3$ &  $(h,e)=(3,6)$ & $9$ & $?$ \\
$37$ &E$3$-E$1$& $X_{2,3}$ & $(g,d)=(0,0)$ & $Q$ &  $(h,e)=(12,12)$ & $8$ & $+$ \\
\hline
 \end{tabular}
 \end{table}
 
 \begin{nt} Most of the notation used in Tables~\ref{table1} and \ref{table2} is self explanatory.
 The column labelled $R$ gathers results from Section~\ref{rationality} on rationality. The symbol $\bullet$ indicates that the link is a known geometric constructions (e.g.~ Cases $17,19$ and $26$), see Section~\ref{examples} for examples and details. The symbol $\times$ indicates that the link is not geometrically realizable. Every numerical Sarkisov link that involves a contraction of type E$1$ with centre along a curve $\Gamma$ such that $(p_a(\Gamma), \deg \Gamma)=(0,0)$ also appears with that contraction replaced by a contraction of type E$3$ or E$4$. I do not repeat these solutions in the tables. 
 \end{nt}
 \begin{rem}
 \label{degree}
 Observe that the possible generators of $\Cl Y/\Pic Y$ have relatively low degree. When $Y$ is the midpoint of a Sarkisov link of type $(\varphi, \alpha)$, one can choose as a generator of $\Cl Y/\Pic Y$ either $\Exc \varphi$ or $\Exc \alpha$ when both $\varphi$ and $\alpha$ are divisorial. When $\varphi$ (resp.~ $\alpha$) is a strict fibration, consider the pullback of $\mathcal{O}_{Z_1}(1)$ (resp.~ $\mathcal{O}_{\widetilde{Z}_1}(1)$) instead of $\Exc \varphi$ (resp.~ $\Exc \alpha$). The anticanonical degree of the generator of $\Cl Y/\Pic Y$ is then given by Lemma~\ref{lem:10} or by the systems in Section~\ref{2raygame}. For $Y\subset \PS^4$ a quartic $3$-fold, the degree of a generator of $\Cl Y/ \Pic Y$ is at most $10$. 
 \end{rem}
 \begin{rem}{(Exclusion of Cases)}
It is known that if $\widetilde{X}\to \PS^2$ is a standard Conic Bundle whose
discriminant has degree at least $6$, then it is not rational. This shows, using Lemma~\ref{ratio}, that the (deformed numerical) Sarkisov
link $9$ in Table~\ref{table1} is not geometrically realizable. 
The other excluded cases correspond to deformed numerical Sarkisov links that are not geometrically realizable \cite{Tak89, IP99}.
\end{rem}

\subsection{Numerical Sarkisov links with centre along higher degree Fano $3$-folds}
\label{tables}
The numerical Sarkisov links with centres along terminal Gorenstein Fano $3$-folds of index $1$ and genus $g\geq 3$ are listed in Table~\ref{table2}.

\begin{center}
  
  \begin{longtable}[p]{ rclclccc}
 
   \label{table2}\\
  \caption{Numerical Sarkisov links, $g\geq4$}\\
  \hline
 & $(\varphi, \alpha)$ & $Z_1$ &$\varphi$ & $\widetilde{Z_1} $& $\alpha$ & $e$ & $R$
  \\[2mm] \hline
  \endfirsthead
  
  \caption[]{Continued}\\
 \hline
 & $(\varphi, \alpha)$ & $Z_1$ &$\varphi$ & $\widetilde{Z_1} $& $\alpha$ & $e$ & $R$
  \\[2mm]  \hline
\endhead
\multicolumn{8}{r}{{Continued on Next Page\ldots}} \\
\endfoot

 \\[2mm] \hline
\endlastfoot

  \multicolumn{8}{c}{$g=4$}\\ \hline
 % $g$ &&$(\varphi, \alpha)$&$Z_1$ &$\varphi$ & $\widetilde{Z_1} $ & $\alpha$ &$e$ &$R$\\ 
    $1$ & E$1$-E$1$&  $X_{22}$  & $(g,d)=(0,7)$ &  $X_{22}$ & $(h,e)=(0,7)$ & $89$ & $+$ \\ 
        $2$ & E$1$-E$1$& $X_{22}$ & $(g,d)=(1,6)$ & $Q$ & $(h,e)=(1,8)$ & $48$  & $+$\\ 
       $3$ & E$1$-E$1$& $X_{22}$ & $(g,d)=(2,9)$ & $X_{14}$ & $(h,e)=(0,3)$ & $12$ & $+$\\ 
       $4$ & E$1$-E$1$& $X_{18}$ & $(g,d)=(1,6)$ & $X_{18}$ & $(h,e)=(1,6)$  & $12$ & $+$ \\ 
       $5$ & E$1$-E$1$& $X_{18}$ & $(g,d)=(0,5)$ & $V_5$ & $(h,e)=(2,9)$  & $47$ & $+$ \\ 
      $6$ & E$1$-dP& $X_{18}$ & $(g,d)=(1,6)$ & $\PS^1$ & $k=6$ & $18$& $+$  \\ 
     $7$ & E$1$-E$1$& $X_{16}$ & $(g,d)=(0,4)$ & $X_{16}$ & $(h,e)=(0,4)$ & $32$ & $+$\\ 
     $8$ & E$1$-E$1$& $X_{16}$ & $(g,d)=(1,5)$ & $\PS^3$ &$(h,e)=(8,9)$ & $8$ &$+$\\  
     $9$ & E$1$-E$1$&  $X_{14}$ & $(g,d)=(0,3)$ & $X_{22}$ & $(h,e)=(2,9)$ & $12$ & $+$\\ 
       $10$ & E$1$-E$1$&  $X_{14}$ & $(g,d)=(1,4)$ & $X_{2,2,2}$ & $(h,e)=(0,0)$  & $4$  & ?\\ 
      $11$ & E$1$-CB&  $X_{14}$ & $(g,d)=(0,3)$ &  $\PS^2$ &$\deg \Delta= 5$ & $23$ &$+$ \\ 
    % $12$&E$1$-E$3$& $X_{14}$ & $(g,d)=(1,4)$ & $X_{2,2,2}$ & & $4$ & $?$ \\
     $12$&E$1$-E$3$& $X_{14}$ & $(g,d)=(1,4)$ & $V_1$ & & $4$ & $?$ \\
    $\bullet 13$& E$2$-E$1$& $X_{14}$&   & $V_{3}$ & $(h,e)= (0,4)$   & $16$& ?\\ 
  $\times14$& E$2$-E$1$& $X_{14}$&   & $Q$ & $(h,e)= (7,10)$   & $15$& $+$\\ 
       $\bullet15$ & E$1$-E$1$&  $X_{12}$ & $(g,d)=(0,2)$ &  $Q$ &$(h,e)=(7,10)$ & $14$ & $+$ \\  
       $\times16$ & E$1$-E$1$&  $X_{12}$ & $(g,d)=(0,2)$ &  $V_3$ &$(h,e)=(0,4)$ & $15$ & $+$ \\  
       $\bullet17$ & E$1$-E$1$& $X_{10}$ & $(g,d)=(0,1)$ & $X_{10}$ & $(h,e)=(0,1)$ & $11$& ? \\ 
        $\times18$& E$1$-E$1$ & $X_{10}$ & $(g,d)=(0,1)$ & $V_5$ & $(h,e)=(6,11)$ & $6$ & $+$\\   
        $19$ & E$3$-E$1$& $X_{2,2,2}/V_1$ & $(g,d)=(0,0)$ & $X_{14}$ & $(h,e)=(1,4)$ & $4$ & ?\\ 
       $20$ & E$3$-dP& $X_{2,2,2}/V_1$ & $(g,d)=(0,0)$ & $\PS^1$ & $k=4$ & $8$ & ? \\ 
        $21$ & E$1$-E$1$& $V_5$ & $(g,d)=(6, 11)$ & $V_5$ & $(h,e)=(0,8)$ & $45$ & $+$ \\ 
       $22$ & E$1$-E$1$& $V_4$ & $(g,d)=(4,8) $  & $\PS^1$ & $k=8$ & $32$ & $+$\\ 
        $23$ & E$1$-E$1$& $V_4$ & $(g,d)=(4,8)$ & $X_{22}$ & $(h,e)=(1,8)$ & $24$ & $+$ \\     
       $24$ & E$1$-E$1$& $V_3$ & $(g,d)=(2,5)$ & $V_4$ &   $(h,e)=(0,6)$ & $28$  & $+$\\ 
        $\times25$ & E$1$-E$1$& $V_2$ & $(g,d)=(0,2)$ & $X_{16}$ & $(h,e)=(0,4)$ & $32$ & $+$\\ \hline
  % $4$ &    $27$ & E$3$-E$1$& $V_1$ & $(g,d)=(0,0)$ & $X_{14}$ & $(h,e)=(1,4)$ & $4$  & ?\\ 
    %$4$ &   $28$ & E$3$-dP& $V_1$ & $(g,d)=(0,0) $  & $\PS^1$ & $k=4$ & $8$ & ?\\ 
    \multicolumn{8}{c}{$g=5$}\\
    \hline
$1$ & E$1$-E$1$& $X_{22}$  & $(g,d)=(0,6)$ &  $X_{22}$ & $(h,e)=(0,6)$ & $36$ & $+$ \\
  $\bullet2$ & E$1$-E$1$& $X_{22}$  & $(g,d)=(1,7)$ &  $\PS^3$ & $(h,e)=(1,7)$ & $14$ & $+$ \\
  $3$ & E$1$-E$1$& $X_{18}$  & $(g,d)=(1,5)$ &  $X_{12}$ & $(h,e)=(0,1)$ & $3$ & $+$ \\ 
   $4$ & E$1$-E$1$& $X_{18}$  & $(g,d)=(0,4)$ &  $Q$ & $(h,e)=(2,8)$ & $20$ & $+$ \\
  $5$ & E$1$-E$1$& $X_{16}$  & $(g,d)=(0,3)$ &  $V_5$ & $(h,e)=(3,9)$ & $12$ & $+$ \\
  $\times 6$& E$2$-E$1$& $X_{16}$&   & $X_{14}$ & $(h,e)= (0,2)$   & $11$& $+$ \\ 
  $\bullet7$& E$2$-E$2$ & $X_{16}$ &&$X_{16}$ && $12$ & $+$\\
   $\times 8$& E$2$-E$2$ & $X_{16}$ &&$V_{2}$ && $12$ & $+$\\ 
   $\bullet9$ & E$1$-E$1$& $X_{14}$  & $(g,d)=(0,2)$ &  $X_{14}$ & $(h,e)=(0,2)$ & $10$ & ? \\ 
 $\times 10$& E$1$-E$2$& $X_{14}$  &  $(g,d)=(0,2)$ &  $X_{16}$&& $11$ & $+$ \\ 
 $\times11$& E$1$-E$2$& $X_{14}$  &  $(g,d)=(0,2)$ &  $V_2$ && $11$ & $+$ \\ 
   $\times12$ & E$1$-E$1$& $X_{12}$  & $(g,d)=(0,1)$ &  $X_{18}$ & $(h,e)=(1,5)$ & $3$  & $+$\\

% \end{tabular}
 %\end{table}
%\begin{table}[h]
%\label{table3}
 % \begin{tabular}{@{} rcccccccc@{}}
 % $g$ &&$(\varphi, \alpha)$&$Z_1$ &$\varphi$ & $\widetilde{Z_1} $ & $\alpha$ &$e$ &$R$\\ 
  $\bullet13$ & E$1$-dP& $X_{12}$  & $(g,d)=(0,1)$ &  $\PS^1$ & $k=5$ & $8$  & $+$\\
 $14$ & E$3$-CB& $X_{10}$ & $(g,d)=(0,0)$ &  $\PS^2$ & $\deg \Delta= 6$ & $6$ & ?\\
  %$5$&& E$3$-CB& $X_{10}$ &  &  $\PS^2$ & $\deg \Delta= 6$ & $6$ & ?\\
   $15$ & E$1$-dP& $V_3$ & $(g,d)=(1,4)$ &  $\PS^1$ & $k=8$ & $24$  & $+$\\
 $16$ & E$1$-dP & $Q$ & $(g,d)=(30,23)$ & $\PS^1$  & $k=2$ & $4464$ & $+$\\
  \hline
  \multicolumn{8}{c}{$g=6$}\\
   \hline
$1$ & E$1$-E$1/2$& $X_{22}$  & $(g,d)=(1,6)$ &  $X_{16}$ & $(h,e)=(0,2)$ &  $2$ & $+$  \\
%$6$ &$2$ & E$1$-E$2$& $X_{22}$  & $(g,d)=(1,6)$ &  $X_{18}$ && $3$ & $+$ \\ 
 $2$ & E$1$-E$1$& $X_{22}$  & $(g,d)=(0,5)$ &  $V_5$ & $(h,e)=(0,7)$ &  $18$& $+$  \\
 $3$ & E$1$-E$1$& $X_{18}$  & $(g,d)=(0,3)$ &  $X_{18}$ & $(h,e)=(0,3)$ &  $9$  & $+$\\
 $\times4$& E$2$-E$1$& $X_{18}$&   & $X_{22}$ & $(h,e)= (1,6)$   & $3$& $+$\\ 
 $\bullet5$& E$2$-dP & $X_{18} $&& $\PS^1$ & $k=6$ & $9$ & $+$\\
 $\times6$ & E$1$-E$1$& $X_{16}$  & $(g,d)=(0,2)$ &  $X_{22}$ & $(h,e)=(1,6)$ &  $2$  & $+$\\
 $\bullet7$ & E$1$-dP& $X_{16}$  & $(g,d)=(0,2)$ &  $\PS^1$ & $k=6$ & $8$ & $+$  \\
 $\bullet8$ & E$1$-CB& $X_{14}$  & $(g,d)=(0,1)$ &  $\PS^2$ & $\deg \Delta=5$  & $6$ & $+$ \\
 $9$& E$3$-E$1$& $X_{12}$  & $(g,d)=(0,0)$ &  $\PS^3$ & $(h,e)=(6,8)$  & $5$& $+$  \\
 $10$& E$1$-CB & $V_4$  & $(g,d)=(2,6)$ &  $\PS^2$ &  $\deg \Delta=2$  & $14$ & $+$ \\
 $\times11$ & E$1$-dP& $V_2$  & $(g,d)=(0,1)$ &  $\PS^1$ & $k=6$ & $8$  & $+$\\
 $\times12$ & E$1$-E$1$& $V_2$  & $(g,d)=(0,1)$ & $X_{22}$ & $(h,e)=(1,6)$ & $2$ & $+$ \\ 
 $13$& E$1$-dP & $Q$ & $(g,d)=(36,33)$ & $\PS^1$  & $k=2,6,8$ & $1620$ & $+$\\
\hline
 \multicolumn{8}{c}{$g=7$}\\
   \hline
   $1$ & E$1$-E$1$& $X_{22}$  & $(g,d)=(0,4)$ &  $X_{22}$ & $(h,e)=(0,4)$ &  $8$ & $+$ \\
 $2$ & E$1$-CB& $X_{18}$  & $(g,d)=(0,2)$ &  $\PS^2$ &$\deg \Delta=4$  & $6$  & $+$\\
 $\bullet3$ & E$1$-E$1$& $X_{16}$  & $(g,d)=(0,1)$ &  $\PS^3$ & $(h,e)=(3,7)$  & $5$ & $+$ \\
 $4$ & E$3$-E$1$& $X_{14}$  & $(g,d)=(0,0)$ &  $Q$ & $(h,e)=(4,8)$ &  $4$  & $+$\\ 
  $5$& E$1$-dP & $Q$ & $(g,d)=(1,14)$ & $\PS^1$  & $k=6$ & $2016$ & $+$\\
 $6$ & E$1$-dP & $Q$ & $(g,d)=(37,38)$ & $\PS^1$  & $k=6$ & $462$ & $+$\\
 \hline
  \multicolumn{8}{c}{$g=8$}\\
   \hline
  $1$ & E$1$-CB& $X_{22}$  & $(g,d)=(0,3)$ & $\PS^2$ & $\deg \Delta=3$ &  $6$ & $+$ \\
 $\bullet2$& E$2$-E$1$& $X_{22}$&   & $\PS^3$ & $(h,e)= (0,6)$   & $6$& $+$\\ 
   $\bullet3$ & E$1$-E$1$& $X_{18}$  & $(g,d)=(0,1)$ &  $Q$ & $(h,e)=(2,7)$ & $4$ & $+$\\
  $4$& E$3$-E$1$ & $X_{16}/V_2$  & $(g,d)=(0,0)$ &  $V_4$ & $(h,e)=(0,4)$  & $4$ & $+$ \\
   $5$ & E$1$-dP& $V_3$  & $(g,d)=(0,2)$ &  $\PS^1$ & $k=8$ &  $8$  & $+$\\
 % $8$& $6$ & E$1$-E$1$& $V_2$  & $(g,d)=(0,0)$ &    $V_4$ & $(h,e)=(0,4)$ &  $4$ & $+$ \\ 
   $6$& E$1$-dP & $Q$ & $(g,d)=(26,30)$ & $\PS^1$  & $k=2$ & $360$ & $+$\\
  \hline
   \multicolumn{8}{c}{$g=9$}\\
   \hline
  $\bullet1$ & E$1$-E$1$& $X_{22}$  & $(g,d)=(0,2)$ &  $Q$ & $(h,e)=(0,6)$ &  $4$ & $+$ \\
  $2$ & E$3$-E$1$& $X_{18}$  & $(g,d)=(0,0)$ &  $V_{5}$ & $(h,e)=(1,6)$ & $3$ & $+$\\
\hline
   \multicolumn{8}{c}{$g=10$}\\
   \hline
     $\bullet1$& E$1$-E$1$ & $X_{22}$  & $(g,d)=(0,1)$ & $V_5$ & $(h,e)=(0,5)$ & $3$  & $+$\\ 
    
  \end{longtable}
\end{center}

\section{A classification of non-factorial terminal Gorenstein Fano $3$-folds}
\label{classification}

Let $Y_4^3 \subset \PS^4$ be a terminal non-factorial quartic $3$-fold. 
Well-known examples of non-factorial quartic
$3$-folds contain planes or quadrics. Yet, a very
general determinantal quartic hypersurface $Y'$ is not factorial and it contains neither a plane nor a quadric. However, $Y'$
does contain a degree $6$ \emph{Bordigo surface}, i.e.~ a surface whose ideal is generated by the $3\times 3$ minors of the matrix defining $Y'$. In the general case, I show that $Y$ contains
some surface of relatively low degree. In other words, the degree of
the surface lying on $Y$ that breaks factoriality cannot be
arbitrarily large.

\subsection{Quartic $3$-folds}
I now prove Theorem~\ref{thm:1}.
 \begin{proof} 
Let $Y$ be a non-factorial terminal Gorenstein Fano $3$-fold and $X\to Y$ a small factorialisation of $X$. I assume that $Y$ does not contain a plane: $X$ is weak-star Fano by Remark~\ref{rem:1}.
We may run a MMP on $X$ as in Theorem~\ref{thm:3}. 

If the MMP on $X$ involves at least one divisorial contraction, then up to a different choice of factorialisation $X\to Y$, we may assume that $X \to X_1$ is divisorial; let $E$ be its exceptional divisor.  The solutions of the systems of Diophantine equations in Section~\ref{2raygame} determine
all the possible contractions $X\to X_1$. To each configuration is associated a Weil non-Cartier divisor $F$ on $Y$.
By Section~\ref{motivation}, $E$ is a rational scroll over a curve $\Gamma$ as in Table~\ref{table1}.

I now assume that the MMP on $X$ involves no divisorial contraction.
 
If any small factorialisation $X \to Y$ is a Conic bundle over $\PS^2, \F_0$ or $\F_1$, we are in Case $5.$ of the Theorem. Hence, it suffices to prove that if $Y$ is the midpoint of a link between two del Pezzo fibrations, then $Y$ contains one of the surfaces listed in the Theorem.

Vologodsky shows that if $Y$ is the midpoint of a link between two nonsingular weak Fano $3$-folds that are extremal del Pezzo
fibrations of degrees $d,d'$, then $d=d'=2$ or $4$ \cite{Vo01}. \cite[Lemma 3.4]{Kal07b} shows that $d\neq2$ because $A_Y$ is very ample. 

\begin{cla}
 If $Y$ is the midpoint of a link between two weak-star Fano dP$4$ fibrations $X$ and $\widetilde{X}$, $Y$
 contains an anticanonically embedded del Pezzo surface of degree $4$,
 and the equation of $Y$ can be written:
\[
Y=\{a_2q +b_2q'=0 \}\subset \PS^4
\]
where $a_2, b_2, q$ and $q'$ are homogeneous forms of degree $2$ on $\PS^4$.  
\end{cla}
Let $F$ be a general fibre of $X \to \PS^1$; $F$ is a
nonsingular del Pezzo surface of degree $4$ and $A_F= {A_X}_{\vert
  F}$. 
  Since $\vert A_X \vert_{\vert F}\subset \vert A_F \vert$, the restriction of the anticanonical map of $Y$ to $F$ factors as  $g_{\vert F} = \nu \circ
\Phi_{\vert A_F \vert}$, where $\nu$ is the projection from a
(possibly empty) linear subspace 
\[
\xymatrix{\PS(H^0(F, A_F))\simeq \PS^4\ar@{-->}[r] & \PS(H^0(F,
  \vert A_{X} \vert_{\vert F}))}.
\]
If $\nu_{\vert F}$ is not the identity, as $h^0(A_F)=h^0(A_X)=5$, the
map $i$ in 
\begin{align*}
0 \to H^{0}(X, A_X-F)\to H^{0}(X, A_X) \stackrel{i} \to H^{0}(F,
A_F)\to \\ \to H^{1}(X, A_X-F)\to 0
\end{align*}
is not surjective, and $H^0(X, A_X-F)\neq (0)$: there is a hyperplane section of $Y$ that contains
$\Phi_{\vert A_X \vert}(F)$. 
As this holds for the general fibre $F$, the fibration $X \to \PS^1$
is induced by a pencil of hyperplanes on $Y$.
Without loss of generality, we may assume that
$\xymatrix{Y \ar@{-->}[r]& \PS^1}$ is determined by the pencil of
hyperplanes $\mathcal{H}_{(\lambda{:} \mu)}= \{\lambda x_0+\mu
x_1=0\}$ for $(\lambda{:} \mu)\in \PS^1$. 
The map $X \to Y$ is a resolution of the base locus of $\mathcal{H}$ on $Y$ and therefore $\Pi= \{x_0{=}x_1{=}0\}= \Bs\mathcal{H}$ lies on $Y$: this contradicts $X$ being weak-star Fano.
 
As $H^0(X, A_X-F)=(0)$, $\nu$ is the identity and $Y$ contains an anticanonically embedded
nonsingular del Pezzo surface $S$ of degree $4$, i.e.~ the intersection
of two quadric hypersurfaces in $\PS^4$. 
Since $S =\{
q{=}q'{=}0\} \subset \PS^4$ lies on $Y$, where $q$ and $q'$ are
homogeneous quadric forms, the equation of $Y$ writes:
\begin{equation}\label{eq:dp4}
Y=\{a_2q +b_2q'=0 \}\subset \PS^4
\end{equation}
with $a_2$ and $b_2$ homogeneous forms of degree $2$. 

Geometrically, the two structures of del Pezzo fibrations on small factorialisations of $Y$ arise as
the maps induced by the
pencils of quadrics (eg $\mathcal{L}=\{ q,q'\}$ and $\mathcal{M}=\{
a_2,b_2\}$) after blowing up their base locus on $Y$, which are
anticanonically embedded del Pezzo surfaces of degree $4$.

Conversely, if the equation of $Y$ is of the form \eqref{eq:dp4} and if $\rk \Cl Y=2$, let $X$ (resp.~ $X'$) be the blow up of $X$ along $S$ (resp.~ along $S'=
\{a_2{=}b_2{=}0\}$), there is a diagram
\[
\xymatrix{
\quad & X \ar[dl] \ar[dr] \ar@{-->}[rr] & \quad & X'\ar[dl] \ar[dr] \\
\PS^1 & \quad & Y & \quad & \PS^1  
}
\]
 The $3$-fold $X$ (resp.~ $X'$) lies on $Q \times \PS^1$ (resp.~ $Q'
 \times \PS^1$) for $Q \subset
 \PS^4$ (resp.~ $Q'$) a quadric that is the proper transform of $\{a_2{=}0 \}$ under
 the blow up of $\PS^4$ along $S$ (resp.~ $S'$). 
The $3$-fold $X$
(resp.~ $X'$) is the
section of a linear system $\vert 2M +2F \vert$ on $ Q \times \PS^1$
(resp.~ $Q'\times\PS^1$), where $M=p_1^{\ast}\mathcal{O}_Q(1)$
(resp.~ $M=p_1^{\ast}(\mathcal{O}_{Q'}(1))$) and $F=p_2^{\ast}\mathcal{O}_{\PS^1}(1)$. 
The map $\xymatrix{X \ar@{-->}[r] & X'}$ is a flop in the curves lying above the points
$\{q{=}q'{=}a_2{=}b_2{=}0\}$.
\end{proof}
\begin{rem}
\label{bound}
The bound on the rank of the divisor class group of quartic $3$-folds given in \cite{Kal07b} is too high: if $Y_4\subset \PS^4$ does not contain a plane, $\rk \Cl Y\leq 6$. Fujita classifies all polarised del Pezzo $3$-folds $(V, L)$ with Cohen-Macaulay Gorenstein singularities \cite{F90}. It is possible that the application of his results would yield an even finer bound. 
\end{rem}
\subsection{Non-factorial terminal Gorestein Fano $3$-folds with $g\geq 4$}
By the same methods as above, one obtains the following theorem for non-factorial terminal Gorenstein Fano $3$-folds of index $1$ and higher genus. 
\begin{thm}
Let $Y=Y_{2g-2}\subset \PS^{g+1}$ be a terminal Gorenstein Fano $3$-fold with $\rho(Y)=1$ and $g(Y)=g$. Then one of the following holds:
\begin{enumerate} 
\item[1.] $Y$ is factorial.
\item[2.]$Y$ contains a plane $\PS^2$ and $g\leq 8$.
\item[3.]  $Y$ is the midpoint of a link between two weak-star Fano del Pezzo fibrations of degree $g+1$ and $g \leq 8$, $g\neq 6$. 
\item[4.] $Y$ has a structure of Conic Bundle over $\PS^2$, $\F_0$ or $\F_2$.
\item[5.] $Y$ contains a rational scroll $E \to C$  over a curve $C$ whose
  genus and degree appear in the appropriate section of Table~\ref{table2} (see page \pageref{table2}).
\end{enumerate}  
\end{thm}
\begin{proof}
This is entirely similar to what is done in the previous subsection. See \cite{Vo01} for $3$.
\end{proof}

\section{Rationality}
\label{rationality}

Classically, it was known that del Pezzo surfaces are rational over any algebraically closed field. 
Understanding whether Fano varieties are rational or not was one of the early problems of higher dimensional birational geometry. Intuitively, Fano varieties can be thought of as being close to $\PS^n$: they are covered by rational curves and, in some sense, these curves should govern their birational geometry. However, the rationality question proved very difficult and  it was not until the early seventies that it was settled for nonsingular Fano hypersurfaces in $\PS^4$ \cite{IM, CG72}.
\cite{IM} developed the Noether-Fano method and proved that any smooth quartic hypersurface is \emph{birationally rigid}-- i.e.~ that every rational map from a smooth quartic hypersurface to a Mori fibre space is a birational automorphism-- and in particular, that quartic hypersurfaces are very far from being rational. This approach was further developed and applied to a number of cases; it yielded surprising rigidity results-- see \cite{S82, Co95, Co00,CPR, Me04, IP99} or the survey \cite{Puk07}. The Noether-Fano method works in principle in any dimension and for singular varieties, but the technical difficulties are considerable. This section presents some results related to the rationality question for terminal Gorenstein Fano $3$-folds. 

\subsection{Rationality, Rational connectivity and ruledness for mildly singular $3$-folds}

\cite{P04} shows that most canonical Gorenstein Fano $3$-folds with Picard rank $1$ that have at least one non-cDV point are rational. These results concern $3$-folds that are strictly canonical. However, one could argue that singularities make Fano $3$-folds ``more rational''.  From the point of view of the Noether-Fano method, the valuations with centre at a singular point give rise to infinitely more complex divisorial extractions--even in the case of isolated hypersurface singularities \cite{Kawk01,Kawk02,Kawk03}-- and hence potentially to many more Sarkisov links and birational maps to other Mori fibre spaces. The following results do not require anything that technical but they do formalise this idea.   

\begin{thm} [Matsusaka's Theorem] \cite[IV.1.6]{Kol96}
\label{mat}
Let $R$ be a DVR with quotient field $K$ and residue field $k$ and denote $T= \Spec R$. Let $f\colon X\to T$ be
a morphism where $X$ is normal and irreducible. 
\begin{enumerate}
\item[1.]
If $X_K$ is ruled over $K$, then $X_k$ has ruled components over $k$.
\item[2.]
If $X_K$ is geometrically ruled, then every reduced irreducible component of $X_k$ is geometrically ruled.
\end{enumerate}
\end{thm} 
\begin{thm}\cite{KMM92a}
\label{rc}
Let $X$ be a normal projective weak Fano $3$-fold. If $X$ is klt, $X$ is rationally connected.
\end{thm}

\begin{lem}
\label{ratio}
Let $f\colon \mathcal{Y}\to \Delta$ be a $1$-parameter smoothing of a terminal Gorenstein Fano $3$-fold $Y$. If $\mathcal{Y}_{\eta}$ is geometrically rational then so is $Y$.
\end{lem}
\begin{proof}
This is a direct consequence of Theorem~\ref{mat}. Indeed, $Y$ is rationally connected by Theorem~\ref{rc}, so that $Y$ is rational if and only if $Y$ is ruled. 
\end{proof}

In now recall and discuss Conjecture~\ref{con:rig}.
\setcounter{con}{0}

\begin{con}
A factorial quartic hypersurface $Y_4\subset \PS^4$ (resp.~ a generic complete intersection $Y_{2,3}\subset \PS^5$) with no worse than terminal singularities has a finite number of models as
Mori fibre spaces, i.e.~ the pliability of  $Y$ is finite.
\end{con}
\begin{rem}\mbox{}
\label{rem11}
\begin{enumerate}
\item[1.]
Conjecture~\ref{con:rig} is supported by some evidence. \cite{Me04} shows that a factorial quartic $3$-fold $Y_4\subset \PS^4$ with ordinary double points is rigid, while \cite{IP96} shows that the same is true for a general non-singular $Y_{2,3}$. Mella's proof is based on the Noether-Fano/maximal singularity method of Iskovskikh-Manin as formulated in \cite{Co00, CPR}. It is difficult to extend these results to terminal Gorenstein singularities, because these methods require a careful analysis of $3$-fold divisorial extractions with centre along (possibly singular) points or curves. While divisorial extractions centred at nonsingular or ordinary double points are reasonably tractable, there is an a priori infinite number of divisorial extractions centred on slightly more complicated singularities \cite{Kawk01,Kawk02,Kawk03}. 
\item[2.] Conjecture~\ref{con:rig} does not hold for some other rigid Fano $3$-folds with Picard rank $1$. For instance, a cubic $3$-fold with a single ordinary double point is both rational and factorial. 
Since several Sarkisov links exist between a nonsingular cubic $3$-fold and a nonsingular Fano $3$-fold $X_{14}\subset \PS^9$ of genus $8$ \cite{IP99,Tak89}, the same phenomenon can be expected on $X_{14}$. 
\item[3.] \cite{ChGr} shows that birational rigidity is not preserved under small deformations, and exhibits a small deformation from a rigid $Y_{2,3}$ with one ordinary double point to a bi-rigid $Y_{2,3}$. Similarly, \cite{CM04} gives an example of a bi-rigid terminal factorial quartic hypersurface. As I mention in the Introduction, in known examples where a birationally rigid Fano $3$-fold $V$ of genus $3$ or $4$ degenerates to a non-rigid and nonrational $3$-fold $V'$, $V'$ has finitely many models as a Mori fibre space, i.e~ $V'$ has finite \emph{pliability}. I believe that the correct notions to consider are rationality on the one hand, and finite pliability on the other.
\end{enumerate}
\end{rem}
\subsection{Rationality of terminal quartic $3$-folds}
\subsubsection{Quartic $3$-folds that do not contain a plane}

Let $Y$ be a non-factorial terminal Gorenstein Fano $3$-fold. Theorem~\ref{thm:1} shows that when $Y$ does not contain a plane, $Y$ has a structure of Conic Bundle, $Y$ is the midpoint of a link between two del Pezzo fibrations of degree $4$, or $Y$ contains a scroll as in Table~\ref{table1}.
Let $X$ be a small factorialisation of $Y$.
\begin{lem}
Let $Y$ be a non-factorial terminal quartic $3$-fold and denote $X\to Y$ a small factorialisation. Assume that the MMP on $X$ involves at least one divisorial contraction. Then $Y$ is rational except possibly if the first divisorial contraction $\varphi$ is one of cases $15,17,25, 29,35$ or $36$ in Table~\ref{table1}.
\end{lem}

\begin{proof}
This is an immediate consequence of the classification of Tables~\ref{table1} and \ref{table2} and of Lemma~\ref{ratio}.
\end{proof}

\begin{lem}
If $\varphi$ is one of cases $15,17,25$ and $29$, and if the MMP on $X$ involves at least another divisorial contraction or a del Pezzo fibration, $Y$ is rational. 
In particular, if $\rk \Cl Y \geq 5$, $Y$ is rational.
\end{lem}

\begin{rem}\mbox{}
\begin{enumerate}
\item[1.]
Note that if $\varphi$ is as in cases $17$ or $36$ and if $\widetilde{Z}_1$ has a singular point, $Y$ is rational.
\item[2.] In Case $29$, when $\rk \Cl Y=2$, the Conic bundle on the deformed Sarkisov link is nonrational \cite{Sh83}. However, it is not clear whether the same is true for $Y$.
\item[3.] According to Conjecture~\ref{con:rig}, one can expect that when $\rk \Cl Y=2$, Case $36$ is impossible, and that when Case $35$ occurs, $Y$ is birationally rigid.  
\item[4.]  It is unlikely that these methods would lead to any conclusion when $\rk \Cl Y=2$ and $Y$ is one of Cases $15,17,25$ or $29$ (see Rem~\ref{rem11}).
\end{enumerate}
\end{rem}

When $X\to \PS^1$ is an extremal del Pezzo fibration, recall the following rationality criteria.

\begin{thm}\cite[Section III.3]{Kol96}
Let $S_k$ be a nonsingular, proper and geometrically irreducible del Pezzo surface of degree $d\geq 5$ over an arbitrary field $k$. Assume that $S(k)\neq \emptyset$, then $S_k$ is rational. 
\end{thm}

\begin{thm}\cite{CT87,KMM92b}
Let $C$ be an algebraic curve defined over an algebraically closed field and let $K=k(C)$ be its field of rational functions. If $X$ is a del Pezzo surface over $K$, then $X(K)\neq \emptyset$ is dense in the Zariski topology of $X$. 
In particular, if $X\to \PS^1$ is a del Pezzo fibration of degree $d\geq 5$, then $X$ is rational.
\end{thm}

\begin{thm}[\cite{Al87, Sh07}]
Let $V \to \PS^1$ be a standard fibration by del Pezzo surfaces of degree $4$. The topological Euler characteristic $\chi(V)$ equals $-8,-4$ or $0$ precisely when $V$ is rational.
\end{thm}
\begin{rem}
In particular, rationality of a del Pezzo fibration $V\to \PS^1$ of degree $4$ is a topological question and depends only on the Hodge numbers of $V$.  \cite{Ch06} shows that if $V\to \PS^1$ is the small factorialisation of a terminal quartic $3$-fold and is nonsingular, then $V$ is nonrational. 
\end{rem}
Last, recall the following rationality criterion for standard Conic Bundles over minimal surfaces.
\begin{thm}\cite{Sh83}
Let $X\to S$ be a standard Conic Bundle over $S= \PS^2$ or $\F_n$. Assume that $\Delta$, the discriminant curve, is connected. If one of the following holds:
\begin{enumerate}
\item[1.] $\Delta+2K_S$ is not effective,
\item[2.] $\Delta\subset \PS^2$ has degree $5$ and the associated double cover $\overline{\Delta}\to \Delta$ has even theta characteristic, 
\end{enumerate}
$X$ is rational.
\end{thm}
\subsubsection{Quartic $3$-folds that contain a plane.}

Assume that  $Y\subset \PS^4$ contains a plane $\Pi=\{x_0{=}x_1{=}0\}$ and let $X$ be the blow up of $Y$ along $\Pi$; $X$ has a natural structure of dP$3$ fibration $\pi\colon X \to \PS^1$ induced by the pencil of hyperplanes that contains $\Pi$ on $\PS^4$ (see \cite[Section 4]{Kal07b} for details).

Write the equation of $Y$ as:
\begin{equation}
\label{eq:2}
\{x_0a_3(x_0, x_1,x_2, x_3,x_4)+ x_1 b_3(x_0, x_1,x_2, x_3,x_4)=0\} \subset \PS^4 
\end{equation}
so that $X$ is given by: 

\begin{eqnarray}
\label{eq:3}
\{t_0a_3(t_0x, t_1x,x_2, x_3,x_4)+ t_1 b_3(t_0x, t_1x,x_2, x_3,x_4)=0\} 
\\
 \subset \PS_{(t_0{:}t_1)}\times \PS(x, x_2, x_3,x_4) .\nonumber
 \end{eqnarray}
 
\begin{lem}\cite[Lemma 4.1]{Kal07b}
The divisor class group $\Cl Y$ is generated by $\pi^{\ast}\mathcal{O}_{\PS^1}(1)$, by the completion of divisors that generate $\Pic X_{\eta}$ and by irreducible components of the reducible fibres of $X$. 
\end{lem}

As $X$ has terminal Gorenstein singularities, \cite{Co96} shows that there is a birational map
\[\xymatrix{X \ar@{-->}[r]^{\Phi} \ar[d] & X' \ar[d]\\
\PS^1 & \PS^1
}\]
where $\Phi$ is the composition of projections from planes contained in reducible fibres and $X'$ has irreducible and reduced fibres. Note that $X_{\eta}\simeq X'_{\eta}$ because $\Phi$ is an isomorphism outside of the reducible fibres of $X$. In particular, if $X_{\eta}$ is rational, $X\to \PS^1$ is geometrically rational, i.e.~ is birational to $\PS^2\times \PS^1$ . 
I recall some results on rationality of cubic surfaces over arbitrary fields.

Let $X_{\eta}$ be a nonsingular cubic surface defined over a field $\eta$ and let $K/\eta$ be a field extension over which the $27$ lines of $X$ are geometric. 
Denote $S_n$ any subset of the $27$ lines on $X_{\eta}\otimes K$ that consists of $n$ skew lines and that is defined over $X_{\eta}$, i.e.~ if $S_n$ contains a line $L$, then it contains all its conjugates under the action of $\Gal(K/\eta)$. Note that, by the geometry of the configuration of the $27$ lines on $X_K$, any $S_n$ has $n\leq 6$.

\begin{thm}\cite{Se42,SD70}
\label{cubics}
\begin{enumerate}\item[1.]
$\overline{NS}(X_{\eta})\otimes_{\Z}\Q$ is generated as a $\Q$-vector space by the class of a hyperplane section of $X_{\eta}$ and by the classes of the $S_n$, when there are any.
\item[2.] If $X_{\eta}$ has an $S_4$ or an $S_5$, $X_{\eta}$ has an $S_2$ or an $S_6$.
\item[3.] If $X_{\eta}$ has an $S_2$, $X_{\eta}$ is rational over $\eta$.
\item[4.] If $X_{\eta}$ has an $S_3$ or an $S_6$ and $X_{\eta}(\eta) \neq \emptyset$, $X_{\eta}$ is rational over $\eta$.
\end{enumerate}
\end{thm}

Here, $X_{\eta}$ is the generic fibre of $X \to \PS^1$, $X_{\eta}$ is a nonsingular cubic surface embedded in $\PS^3$ over $\C(t)$, with coordinates $x, x_2,x_3,x_4$ (see \eqref{eq:3}). 
 
 \begin{cla}
 Assume that $X_{\eta}$ contains a Cartier divisor of type $S_n$ and denote 
 $D_n$ the completion of $S_n$ to a (Weil) divisor on $X$. The proper transform of $D_n$ on a small factorialisation of $X$ has anticanonical degree $n$; the image of $D_n$ on $Y$ is Weil non-Cartier.
 \end{cla} 
In the light of Theorem~\ref{cubics}, it is then natural to consider the following cases:
\setcounter{case}{0}

\begin{case}{$X$ is an extremal Mori fibre space, i.e.~ $\rk \Cl Y =2$, $X\to \PS^1$ has irreducible and reduced fibres  and $\rho(X_{\eta})=1$.}
\end{case}  
It is known that $X$ admits another model as a Mori fibre space \cite{BCZ}. Indeed, $Y$ is the midpoint of a link 
\[
\xymatrix{
\quad & X \ar@{-->}[rr] \ar[dl]\ar[dr] &&\widetilde{X}\ar[dr]\ar[dl]& \quad\\
\PS^1 & \quad& Y & \quad & Z}
\]
where $Z=Y_{3,3}\subset \PS(1^5,2)$ is a codimension $2$ terminal Fano $3$-fold with one point of Gorenstein index $2$ at $P=(0{:}0{:}0{:}0{:}0{:}0{:}1)$; $Y\dashrightarrow Z$ can be described as follows. Introduce a variable of weight $2$  \[y= \frac{a_3}{x_1}=\frac{b_3}{x_0},\] then $Z$ is the complete intersection:
\[\left \{ \begin{array}{c}
a_3-yx_1=0\\
b_3-yx_0=0  
\end{array}\right.
\]
The contraction $\widetilde{X}\to Z$ contracts the preimage of the plane $\{x_0{=}x_1{=}0\}$ to the point $P$, the map $X\dashrightarrow \widetilde{X}$ is the flop of the rational curves lying above the locus $\{ x_0{=}x_1{=}a_3{=}b_3{=}0\}$, and $\widetilde{X} \to Y$ is the blow up of the surface $\{a_3{=}b_3{=}0\}$.
Recall that $X$ is a section of the linear system $ \vert 3M+L\vert $ on the scroll $\F(0,0,1)$ (see \cite{BCZ} for notation conventions on scrolls); \cite{Ch08} shows that if $X$ is a general member of $\vert 3M+L \vert$,  $X$ is nonrational. I make the following conjecture: 
\begin{con}
If $X$ is a standard dP3 fibration, $X$ is bi-rigid. 
\end{con}

\begin{case}{$X$ is not an extremal Mori fibre space, i.e.~ $\rk \Cl Y \geq 3$, and $\rho(X_{\eta})=1$.}
\end{case}
\begin{lem}{\cite{Kol96}}
\label{3planes}
Let $Y_4\subset \PS^4$ be a quartic hypersurface. If $Y$ contains three planes $\Pi_0, \Pi_1, \Pi_2$ such that $\Pi_0\cap\Pi_1\cap\Pi_2=\emptyset$, $Y$ is rational. 
\end{lem}

\begin{cor} Let $X\to Y$ be as above. Assume that $\rho(X_{\eta})=1$, 
if there are at least $3$ planes lying in at least $2$ distinct reducible fibres of $X$, $Y$ is rational. More precisely, if $X$ has either at least two reducible fibres, one of which is the union of $3$ planes or if $X$ has at least $3$ reducible fibres, $Y$ is rational. 
\end{cor}
\begin{proof}
This follows from the possible configurations of planes lying in reducible fibres obtained as in \cite[Section 4]{Kal07b}.
\end{proof}

Assume that $X\to \PS^1$ has $\rho(X_{\eta})=1$, and that $X\to \PS^1$ has $1$ or $2$ reducible fibres, each containing a quadric ($\rk \Cl Y=3$ or $4$). 
Among the generators of $\Cl Y/\Pic Y$, there is a surface $S$ such that $A_Y^2\cdot S= 2$, i.e.~ there is a quadric lying on $Y$. Denote $f\colon \widetilde{X}\to X \to Y$ a small factorialisation of $X$ and $Y$ and note that there is an extremal divisorial contraction $\varphi \colon \widetilde{X} \to \widetilde{X}_1$ such that $\widetilde{S}=f_{\ast}^{-1}S=\Exc \varphi$ (possibly after flops of $\widetilde{X}$). 

Observe that $\widetilde{X}_1$ is the small modification of a terminal Gorenstein Fano $3$-fold $Y_1=Y_{2,3}\subset \PS^5$. For any divisor $D_1\subset \widetilde{X_1}$, the proper transform $D$ of $D_1$ on $\widetilde{X}$ is such that $A_{Y}^2\cdot D\leq A_{Y_1}^2\cdot D_1$ and the inequality is strict when $D$ intersects the quadric $S$ (see the proof of \cite[Theorem 3.2]{Kal07b}).
Note that $\Pi$ and all planes contained in reducible fibres of $X\to \PS^1$ do intersect  the quadric $S$ and since $\rho(X_{\eta})=1$, $\widetilde{X_1}$ is weak-star Fano: the methods of the previous subsection apply.

More precisely, as $Y_1$ is terminal Gorenstein and has $\rk \Cl(Y_1)\geq 2$, unless $Y_1$ has a structure of Conic Bundle or the MMP on $\widetilde{X_1}$ consists of one divisorial contraction of type $10,11,12,13,14,18,20$ or $21$ in Table~\ref{table2}, $Y$ is rational.
 
\begin{exa}
In particular, this gives potential examples of rational cubic fibrations that are not geometrically rational.
\end{exa}

\begin{case} {$X$ is not an extremal Mori fibre space, i.e.~ $\rk \Cl Y \geq 3$, and $\rho(X_{\eta})>1$.}
\end{case}

\begin{pro}
If $\rho(X_{\eta})\geq 3$, $X$ is rational.  If $\rho(X_{\eta})=2$ and either $\Cl Y/\Pic Y$ is not generated by planes or $X$ has at least one reducible fibre, $X$ is rational.
\end{pro}
\begin{proof}

Theorem~\ref{cubics} shows that unless $\Pic X_{\eta}$ is generated by the class of a hyperplane section and divisors of type $S_1$, $X_{\eta}$ is rational. But then, as \cite{Co96} shows that $X \to \PS^1$ is birational to a cubic fibration $X'\to \PS^1$ with reduced and irreducible fibres and $X_{\eta}\simeq X'_{\eta}$, $X\to \PS^1$ is rational.

We now turn to the case when $\rho(X_{\eta})>1$ and $X_{\eta}$ does not contain any $S_n$ for $n\geq 2$.

The proposition follows from the following Claims.
\begin{cla}
If $\rho(X_{\eta})\geq 3$ and if $\Pi', \Pi''$ are two planes on $X\to \PS^1$ that arise as completions of divisors of type $S_1$ on $X_{\eta}$, then $\Pi\cap\Pi'\cap\Pi''= \emptyset$.
\end{cla}
Any $S_1$ lying on $X_{\eta}$ completes to a plane $\Pi'$ that meets $\Pi$ in a point. 
Indeed, if $\Pi$ and $\Pi'$ met in a line, the image of $\Pi'$ on $Y$ would be contained in a hyperplane section of the original quartic $Y$, and $\Pi'$ would have to be contained in a reducible fibre. If $X_{\eta}$ contains two distinct $S_1$, these cannot be skew (otherwise they would form an $S_2$) and therefore up to coordinate change on $\PS(x, x_2, x_3,x_4)$, $X_{\eta}$ contains the lines
\begin{eqnarray*}
L=\{x_2{=}x_3{=}0\}\\
L'=\{x_2{=}x_4{=}0\}
\end{eqnarray*}
 so that $Y$ contains the planes $\{x_0{=}x_1{=}0\}$,  $\{x_2{=}x_3{=}0\}$ and $\{x_2{=}x_4{=}0\}$ and by Lemma~\ref{3planes}, $Y$ is rational.

\begin{cla}
If there are at least $3$ planes lying in reducible fibres of $X\to \PS^1$ then we may choose $\Pi''$ lying in a reducible fibre of $X$ such that $\Pi\cap\Pi'\cap\Pi''= \emptyset$.
\end{cla}
Since any plane contained in a fibre of $X\to \PS^1$ intersect $\Pi$ in a line and that given any $3$ such planes, \cite{Kal07b} shows that the $3$ associated lines are distinct and non-concurrent, we may choose one plane that does not contain $\Pi\cap\Pi'$. 
\end{proof}

We have proved the following.
\begin{pro}
Let $Y_4\subset \PS^4$ be a quartic hypersurface that contains a plane. If $6\leq \rk \Cl Y\leq 16$, $Y$ is rational. 
\end{pro}

\section{Examples and geometric realizability of numerical Sarkisov links}
\label{examples}
\subsection{Examples}

In this section, I construct examples of non-factorial Fano $3$-folds with terminal Gorenstein singularities. I use the Tables of numerical Sarkisov links to recover some known examples and construct some new ones. 
\begin{exa}
Let $X_{2g-2}\subset \PS^{g+1}$ be a nonsingular Fano $3$-fold of genus $g \geq 7$ and let $P\in X$ be a point that does not lie on any line of $X$ (such a point exists by \cite{Isk78}). Let $\widetilde{X}\to X$ be the blow up of $P$. Then $\widetilde{X}$ is a weak-star Fano $3$-fold with Picard rank $2$. The anticanonical model $Y$ of $\widetilde{X}$ is a terminal Gorenstein non-factorial Fano $3$-fold of genus $g-4$. The map $\widetilde{X}\to Y$ is small and contracts the preimages of conics through $P$ to points (\cite{Tak89} proves that there are finitely many such conics).
\begin{enumerate}
\item[1.] When $g=7$, $Y$ is a quartic $3$-fold that is the midpoint of a link where both contractions of the Sarkisov link are of type E$2$. The link is a self-map of $X_{12}\subset \PS^9$; the centre of the link is a rational quartic $3$-fold $Y$ with $\rk \Cl Y=2$.
\item[2.] When $g\geq 8$, \cite{Tak89} lists all possible constructions starting with a nonsingular Fano $3$-fold $X_{2g-2}\subset \PS^{g+1}$. Takeuchi uses Hodge theoretical computations to show that some numerical Sarkisov links are not realizable. Here, since I allow terminal Gorenstein singularities, it is not clear that these links can be excluded (see Remark~\ref{Hodge}).   
\end{enumerate}
\end{exa}
Let $X$ be a nonsingular Fano $3$-fold with $\rho(X)=1$  and $\Gamma \subset X$ a curve such that $X=Z_1$ and $\Gamma$ is the centre of $\varphi$ for one case appearing in Table~\ref{table1} (resp.~ of Table~\ref{table2}).
Let $\widetilde{X} \to X$ be the blow up of $X$ along $\Gamma$. By construction, $A_{\widetilde{X}}^3=4$ (resp.~ $2g-2$ for $g\geq 4$), so that if $A_{\widetilde{X}}$ is nef, it is big and  $\widetilde{X}$ is a Picard rank $2$ weak Fano $3$-fold. Observe that when $\Gamma$ is an intersection of members of $\vert A_X\vert$, then $A_{\widetilde{X}}$ is nef and big.

The anticanonical map $f \colon \widetilde{X} \to Y$ maps to a Gorenstein Fano $3$-fold with canonical singularities. If, in addition, $(A_{\widetilde{X}})^2\cdot D>0$ for every effective divisor $D$, $Y$ has terminal singularities and $f$ is small. Still by construction, in this case, $f$ is not an isomorphism because  $e\neq 0$, and $Y$ is a non-factorial terminal Gorenstein Fano $3$-fold with $\rho(Y)=1$, $\rk \Cl Y=2$.

\begin{thm}\cite{Sh79, Re80, Tak89, IP99}\label{lineconic}
Let $X_{2g-2}\subset \PS^{g+1}$ be a nonsingular anticanonically embedded Fano $3$-fold of index $1$. If $g\geq 5$ (resp.~ $g\geq 6$), there exists a line (resp.~ a smooth conic) on $X$. 
For any $g\geq 5$, if $X$ contains a line and a smooth conic, it also contains a rational normal cubic curve.
\end{thm}

\begin{exa}\cite{ Isk78, IP99}
Let $X=X_{2g'-2}\subset \PS^{g'+1}$ be a nonsingular (or more generally terminal Gorenstein factorial) Fano $3$-fold of index $1$ such that $A_X$ is very ample and let $\Gamma$ be a line lying on $X$.  As above, let $\widetilde{X}\to X$ be the blow up along $\Gamma$ and let $\widetilde{X}\to Y$ be the anticanonical map of $\widetilde{X}$. Recall the following result of Iskovskikh's: 
\begin{thm}\cite{Isk78}
If $\Gamma \subset X_{2g'-2}$ is a line on a nonsingular Fano $3$-fold of genus $g'$ and Picard rank $1$, and if $\widetilde{X}\to X$ is the blow up of $X$ along $\Gamma$, then $\widetilde{X}$ is a small modification of a terminal Gorenstein Fano $3$-fold $Y_{2g-2}$ of index $1$, Picard rank $1$ and genus $g= g'-2$. 
\end{thm}

By construction, $Y$ is not factorial and its divisor class group is generated by the hyperplane section and by a surface $\overline{E}=f(\PS(\mathcal{N}^v_{\Gamma/X}))$, which is the image by the anticanonical map of a cubic scroll. 

 The blow up $f$ is one side of a Sarkisov link with midpoint along $Y$. Note that the rational map between the two sides of the Sarkisov link $Z_1=X_{2g'-2}\dashrightarrow \widetilde{Z}_1$ is Iskovskikh's double projection from a line \cite{Isk78}, that enabled him to classify Fano $3$-folds of the first species.
\begin{enumerate} 
\item[1.]Case $30$ in Table~\ref{table1} is a geometric construction that was known classically \cite{Be77, BCZ}. 
Let $X=X_{2,2,2}\subset \PS^6$ be a codimension $3$ complete intersection of quadrics in $\PS^3$ and $l\subset X$ be a line. Then the other contraction in the link starting with the projection from $l$ is a conic bundle with discriminant of degree $7$.  Conversely, given a plane curve $\Delta\subset \PS^2 $ of degree $7$, \cite{BCZ} constructs standard conic bundles with ramification data a $2$-to-$1$ admissible cover $N \to \Delta$.  When $\deg \Delta=7$, there are $4$ deformation families of standard conic bundles and Case $30$ corresponds to the generic even theta characteristic case. 
By \cite{Sh83}, the standard conic bundle $\widetilde{X}$ is non rational.
\item[2.]\cite{Isk78}
When $X$ is nonsingular, the link that occurs is Case $17, g=4$ in Table $2$ for $g'=6$, Case $13, g=5$ for $g'=7$, Case $8,g=6$ for $g'=8$, Case $3,g=7$ for $g'=9$, Case $3, g=8$ for $g'=10$ and Case $1,g=10$ for $g'=12$. 
 \end{enumerate}
Note that Case $31$ in Table~\ref{table1}, and Cases $18, g=4$, $12, g=5$, and Cases $11,12, g=6 $do not occur \cite{Isk78, IP99} if $Z_1$ is nonsingular. 
One can describe explicitly the inverse rational map $\widetilde{Z}_1 \dashrightarrow Z_1$ for $g \geq 7$ by choosing the curve $C$ carefully on an appropriate $\widetilde{Z_1}$ \cite{IP99}.
\end{exa}

\begin{exa}\cite{Tak89}
Let $X=X_{2g'-2}\subset \PS^{g'+1}$ be a nonsingular (or more generally terminal Gorenstein factorial) Fano $3$-fold of index $1$ such that $A_X$ is very ample and let $\Gamma$ be a smooth conic lying on $X$.  As above, let $\widetilde{X}\to X$ be the blow up along $\Gamma$ and let $\widetilde{X}\to Y$ be the anticanonical map of $\widetilde{X}$. 
\begin{thm}\cite{Tak89}
Let $\Gamma \subset X_{2g'-2}$ be a conic on a nonsingular Fano $3$-fold of genus $g'$ and Picard rank $1$, and $\widetilde{X}\to X$ the blow up of $X$ along $\Gamma$. If $g'\geq 7$ and $\Gamma$ is general, then $\widetilde{X}$ is a small modification of a terminal Gorenstein Fano $3$-fold $Y_{2g-2}$ of index $1$, Picard rank $1$ and genus $g= g'-3$. If $g'\geq 9$, the same holds for any conic $\Gamma \subset X$.
\end{thm}
Note that when $\widetilde{X}$ is a small modification of a terminal Gorenstein Fano $3$-fold $Y$ with $\rho(Y)=1$, the divisor class group of $Y$ is generated by a hyperplane section and by a surface $\overline{E}= f(\PS(\mathcal{N}_{\Gamma/X}^v))$ which has degree $4$. More precisely, $\mathcal{N}_{\Gamma/X}= \mathcal{O}_{\PS^1}(d)\oplus\mathcal{O}_{\PS^1}(-d)$, for $d=0,1$ or $2$.
If $d=0$, $f_{\vert E}\colon \PS^1 \times \PS^1\to \overline{E}$ is induced by a divisor of bidegree $(1,2)$.
If $d=1$, $f_{\vert E}\colon \F_2 \to \overline{E}$ is induced by $\vert s+3f\vert$ and if $d=2$, $f_{\vert E}\colon \F_4 \to \overline{E}$ is induced by $\vert s+4f\vert$.

The blow up $f$ is one side of a Sarkisov link with midpoint along $Y$ and corresponds to one of Cases $25$ or $26$ in Table~\ref{table1}, or Cases $15,16$ for $g=4$, $9,10,11$ for $g=5$, $6,7$ for $g=6$, $2$ for $g=7$ or $1$ for $g=9$. \cite{Tak89} shows that for nonsingular Fano $3$-folds $X_{2g-2}=Z_1$, the Cases indicated by $\bullet$ in Tables~\ref{table1} and \ref{table2} are the only geometrically realizable constructions.
\end{exa}
\begin{exa} Let $X_{2g'-2}\subset \PS^{g+1}$ be a nonsingular Fano $3$-fold with $g'\geq 6$. By Theorem~\ref{lineconic}, there is a rational normal cubic curve $\Gamma$ lying on $X$. If $\widetilde{X}\to Y$ is small, then if $g'=7$, we are in case $18$ or $19$ of Table~\ref{table1}, $Y$ is a terminal Gorentsein factorial quartic $3$-fold that is rational; and if $g'\geq 8$, we are in one of Cases $9$ or $11, g=4$, $5, g=5$, $3, g=6$ or $1,g=8$ of Table~\ref{table2}.  
\end{exa}
\begin{thm}\cite{Mo84, MM83}
\label{exi}
Let $k=\overline{k}$ be a field of characteristic $0$ and let $d>0$ and $g\geq 0$ be integers.
There exists a nonsingular curve $C$ lying on a nonsingular quartic surface $S_4\subset \PS^3_k$  with $(p_a C, \deg C)=(g,d)$ if and only if 
\[g=d^2/8+1\mbox{ or }  g<d^2/8\] and $(g,d)\neq (3,5)$. 
\end{thm}

\begin{exa}
I now use Theorem~\ref{exi} to show that some constructions that appear in Tables~\ref{table1} and \ref{table2} may be geometrically realizable.

\begin{enumerate}
\item[1.]
Let $C\subset \PS^3$ be a nonsingular curve that is an intersection of nonsingular quartic surfaces with $(p_a C, \deg C)= (15,11)$. Let $X$ be the blow up of $\PS^3$ along $C$, assume that $X$ is the small modification of  a terminal quartic hypersurface $Y \subset \PS^4$. The linear system $\vert \mathcal{O}_{\PS^3} \vert $ determines a rational map $\PS^3 \dashrightarrow X_{10}\subset \PS^7$ that corresponds to the inverse of Case $28$ in Table~\ref{table1}. The midpoint $Y_4\subset \PS^4$ is a non-factorial rational quartic $3$-fold; $\Cl Y$ is generated by the hyperplane section and the image in $Y$ of $E$ or $D$. Note that this rational map provides an example of a rational Fano $3$-fold of genus $6$.
\item[2.]\cite{IP99}
Let $C\subset \PS^3$ be a nonsingular curve that is an intersection of nonsingular quartic surfaces with $(p_a C, \deg C)= (7,9)$. Let $X$ be the blow up of $\PS^3$ along $C$, then since $C$ is an intersection of nonsingular quartic surfaces, $X$ is the small modification of  a terminal quartic hypersurface $Y \subset \PS^4$. The linear system $\vert \mathcal{O}_{\PS^3}(15)-4C \vert$ determines a rational map $\PS^3 \dashrightarrow X_{12}\subset \PS^8$ that corresponds to the inverse of Case $19$ in Table~\ref{table1}. The midpoint $Y_4\subset \PS^4$ is a non-factorial rational quartic $3$-fold; $\Cl Y$ is generated by the hyperplane section and the image in $Y$ of $E$ or $D$.
\item[3.]
Let $C\subset \PS^3$ be a nonsingular curve that is an intersection of nonsingular quartic surfaces  with $(p_a C, \deg C)= (3,8)$. Let $X$ be the blow up of $\PS^3$ along $C$, assume that $X$ is the small modification of  a terminal quartic hypersurface $Y \subset \PS^4$. The linear system $\vert \mathcal{O}_{\PS^3}\vert$ determines a rational map $\PS^3 \dashrightarrow V_3\subset \PS^4$ that corresponds to the inverse of Case $34$ in Table~\ref{table1}. The midpoint $Y_4\subset \PS^4$ is a non-factorial rational quartic $3$-fold; $\Cl Y$ is generated by the hyperplane section and the image in $Y$ of $E$ or $D$. Note that in this case $V_3\subset \PS^4$ would necessarily be singular, because it would be rational. 
\item[4.]
Let $C\subset \PS^3$ be a nonsingular curve that is an intersection of nonsingular quartic surfaces with $(p_a C, \deg C)= (1,7)$. Let $X$ be the blow up of $\PS^3$ along $C$, assume that $X$ is the small modification of  a terminal Gorenstein Fano $3$-fold $Y_{2,2,2} \subset \PS^6$ that is non-factorial and rational. The linear system $\vert \mathcal{O}_{\PS^3}\vert$ determines a rational map $\PS^3 \dashrightarrow X_{22}$ that corresponds to the inverse of Case $2, g=5$ in Table~\ref{table2}. 
\item[5.]
Let $C\subset \PS^3$ be a nonsingular curve lying that is an intersection of nonsingular quartic surfaces  with $(p_a C, \deg C)= (6,8)$. Let $X$ be the blow up of $\PS^3$ along $C$, assume that $X$ is the small modification of  a terminal Gorenstein Fano $3$-fold $Y_{10} \subset \PS^7$ that is non-factorial and rational. The linear system $\vert \mathcal{O}_{\PS^3}\vert$ determines a rational map $\PS^3 \dashrightarrow X_{22}$ that corresponds to the inverse of Case $9,g=6$ in Table~\ref{table2}. 
\end{enumerate}
\end{exa}
\begin{exa}\cite{IP99}
Let $C\subset \PS^3$ be a nonsingular non hyperelliptic curve lying on a nonsingular quartic surface with $(p_a C, \deg C)= (3,7)$. Let $X$ be the blow up of $\PS^3$ along $C$, $X$ is the small modification of  a terminal Gorenstein Fano $3$-fold $Y_{12} \subset \PS^8$,= that is non-factorial and rational. The linear system $\vert \mathcal{O}_{\PS^3}\vert$ determines a rational map $\PS^3 \dashrightarrow X_{16}$ that corresponds to the inverse of Case $3,g=7$ in Table~\ref{table2}. 
\end{exa}

\subsection{Some remarks on geometric realizability}

Classically, it has been shown that numerical Sarkisov links were not geometrically realizable by using constraints on the Hodge numbers of blow ups of nonsingular varieties along smooth centres or constraints on Euler characteristics of fibrations. In the case of divisorial contractions of factorial terminal Gorenstein $3$-folds, I have been unable to extend these results so as to use them to rule out some numerical Sarkisov links.
 It is easy to show the following weakened version:    
\begin{lem}
\label{htexc}
Let $Z$ be a nonsingular weak Fano $3$-fold and $\varphi \colon Z\to Z_1$ an extremal divisorial contraction with centre along a curve $\Gamma$ and such that $Z_1$ is a terminal Gorenstein Fano $3$-fold. 
Then
\[
h^{1,2}(Z)\leq h^{1,2}(\mathcal{Z}_{1, \eta})+ p_a(\Gamma),
\]
where $p_a(\Gamma)$ denotes the arithmetic genus of $\Gamma$ and $\mathcal{Z}_1$ a smoothing of $Z_1$. 
\end{lem}

%\begin{proof}
%The map $Z\to Z_1$ is a proper modification, so that the diagram
%\[
%\xymatrix{ E \ar[d]\ar[r]  & Z \ar[d]\\
%\Gamma \ar[r] & Z_1
%}
%\] is of cohomological descent and 
%\[
%\cdots \to H^k(Z_1, \Q) \to H^k(\Gamma, \Q) \oplus H^k(Z, \Q)\to H^k(E, \Q) \to H^{k+1}(Z_1, \Q)
%\]
%is a long exact sequence of mixed Hodge structures. Note that $Z, Z_1, E$ and $\Gamma$ are all compact and $Z$ is smooth, so that  $\Gr^W_j H^k=(0)$ for $j>k$, where $H^k$ denotes the cohomology of $Z_1, E$ or $\Gamma$ and $\Gr^W_kH^k(Z)= H^k(Z)$; in particular, the sequence
%\[
%0\to \Gr^W_3 H^3(Z_1)\to H^3(Z) \to \Gr_3^W H^3(E) \to \Gr_4^W H^3(Z_1) \to 0
%\]
%is exact. Note that by \cite{Cut88}, $Z_1$ is factorial. Let $\mathcal{Z}_1\to \Delta$ denote a smoothing of $Z_1$, and $\Sing Z_1=\{P_1, \cdots P_n\}$, then
%\[
%h^3(Z_1)=h^3(\mathcal{Z}_{1, \eta})- \sum^n_{i=1} h^3(B_{i, \eta}),
%\]
%where $B_i$ denotes the vanishing cycle associated to $P_i$.
 
%The surface $E$ is a $\PS^1$-bundle over $\Gamma$ as can be seen directly by writing down the equation in an affine chart of the blow up of $\Gamma$ near a singular point, because the singularities of $\Gamma$ are locally planar, i.e.~ of the form $\{x=f(y,z)=0\}$. In particular, by the Leray-Hirsch spectral sequence, $\Gr^W_3 H^3(E) \simeq H^1(\Gamma)$ is an isomorphism of weight $(1,1)$.  

%This shows that \[h^3(Z)\leq h^3(Z_1)+\dim \Gr^W_3 H^3(E)\leq h^3({\mathcal{Z}}_{1,\eta})+ h^1(\Gamma),\] and since $\Gamma$ is a l.c.i. curve and hence is Gorenstein, this is equivalent to
%\[
%h^{2,1}(Z)\leq h^{2,1}(\mathcal{Z}_{1, \eta})+ p_a(\Gamma).
%\]
%\end{proof}

\begin{rem}
\label{Hodge}
\cite{Kol89} shows that $Z$ and $\widetilde{Z}$ have the same analytic type of singularities. Since $h^{1,2}(Z)$ and $h^{1,2}(\widetilde{Z})$ can be expressed only in terms of local invariants of singularities and of $\rk \Cl Y$, where $Y$ is the anticanonical model of $Z$ and $\widetilde{Z}$, $h^{1,2}(Z)=h^{1,2}(\widetilde{Z})$ . In order to exclude some numerical Sarkisov links, I would need to find a lower bound for $h^{1,2}(Z)$ (resp.~ $h^{1,2}(\widetilde{Z})$).  
This would follow if the following question could be answered.
\begin{que}
Is it possible to relate $W_3H^4(Z)$ and $W_2H^3(Z)$ when $Z$ has terminal Gorenstein singularities? What if $Z$ is factorial?
\end{que}
\end{rem}
\begin{rem}
Observe that in order to determine that a numerical Sarkisov link is not realizable, it is enough to observe that no deformed (nonsingular) link exists between a Fano in the deformation family of $Z_1$ and a Fano in the deformation family of $\widetilde{Z}_1$. This has been used  in the previous subsections.
\end{rem}

Another question of interest would be to understand the geometric meaning of the correction term $e$ that appears in the tables of numerical Sarkisov links. The proof of Lemma~\ref{lem:4} shows that $e$ is the intersection of $E$, the exceptional divisor of the left hand side contraction, with the flopping locus of $Z\dashrightarrow \widetilde{Z}$.  The large values of $e$ that appear in the table suggest the following question.
\begin{que}
Let $f \colon Z\to Y$ be a small factorialisation and assume that $f^{-1}(P)$ is a chain of rational curves $\cap \Gamma_i$. If $E$ is the proper transform on $X$ of a Weil non-Cartier divisor passing through the singular point $P$, is it possible to have $E\cdot \Gamma_i>1$? The surface $E$ is a priori not Cohen Macaulay at $P$, but is it possible to bound this intersection number? 
\end{que}

\bibliography{bib}
\end{document}